\documentclass[11pt]{amsart}
\usepackage{amsmath,amsthm,amssymb,enumerate,mathscinet,mathtools}
\usepackage{fullpage}
\usepackage{graphicx,xcolor,pgfplots}
\usepackage{verbatim}
\usepackage{mathrsfs}
\usepackage{hyperref}
\usetikzlibrary{topaths,calc} 
\tikzstyle{vertex} = [fill,shape=circle,node distance=80pt]
\tikzstyle{edge} = [opacity=0.4,fill opacity=0.0,line cap=round, line join=round, line width=40pt]
\tikzstyle{elabel} =  [fill,shape=circle,node distance=30pt]
\newcommand{\msc}[1]{\begin{center}MSC2010: #1.\end{center}}
\newtheorem{theorem}{Theorem}[section]

\newtheorem{lemma}[theorem]{Lemma}

\newtheorem{conjecture}[theorem]{Conjecture}
\newtheorem{claim}[theorem]{Claim}

\newtheorem{thm}[theorem]{Theorem}

\newtheorem{cor}[theorem]{Corollary}
\newtheorem{problem}[theorem]{Question}

\numberwithin{equation}{section}

\DeclareMathOperator{\rank}{rank}

\DeclareMathOperator{\polylog}{polylog}

\def\COMMENT#1{}

\allowdisplaybreaks

\title{An asymmetric random Rado theorem for single equations: the $0$-statement}

\author{Robert Hancock and
 Andrew Treglown}

\thanks{RH: Institut f\"ur Informatik, Heidelberg University, Im Neuenheimer Feld 205, 69120, Heidelberg, Germany. Previous affiliation: Faculty of Informatics, Masaryk University, Botanick\'a 68A, 602 00 Brno, Czech Republic, {\tt hancock@informatik.uni-heidelberg.de}. This work has received funding from the European Research Council (ERC) under the European Union's Horizon 2020 research and innovation programme (grant agreement No 648509) and from the MUNI Award in Science and Humanities of the Grant Agency of Masaryk University. This publication reflects only its authors' view; the European Research Council Executive Agency is not responsible for any use that may be made of the information it contains.
AT: University of Birmingham, United Kingdom, {\tt a.c.treglown@bham.ac.uk}.}

\begin{document}

\begin{abstract}
A famous result of Rado characterises those integer matrices $A$ which are partition regular, 
i.e. for which any finite colouring of the positive integers gives rise to a monochromatic solution to the equation $Ax=0$. 
Aigner-Horev and Person recently stated a conjecture on the probability threshold for
the binomial random set $[n]_p$ having the asymmetric random Rado property:
given partition regular matrices $A_1, \dots, A_r$ (for a fixed $r \geq 2$), 
however one $r$-colours $[n]_p$, 
there is always a colour $i \in [r]$ such that there is an $i$-coloured solution to $A_i x=0$.
This generalises the symmetric case, which was resolved by
R\"odl and Ruci\'nski, and Friedgut, R\"odl and Schacht.
Aigner-Horev and Person proved the $1$-statement of their asymmetric conjecture. 
In this paper, we resolve the $0$-statement in the case where the $A_i x=0$ correspond to  single linear equations. Additionally we close a gap in the original proof of the 0-statement of the (symmetric) random Rado theorem.
\end{abstract}

\date{\today}

\maketitle
\msc{ 5C55, 5D10, 11B75}

\section{Introduction}

An important branch of arithmetic Ramsey theory concerns partition properties of sets of integers. 
A cornerstone result in the area is Rado's theorem~\cite{rado} which characterises all those systems of homogeneous linear equations $\mathcal L$ for which every finite colouring of $\mathbb N$ yields a monochromatic solution to $\mathcal L$.
Note that this provides a wide-reaching generalisation of other classical results in the area such as Schur's theorem~\cite{schur} (i.e. when  $\mathcal L$ corresponds to $x+y=z$) and van der Waerden's theorem~\cite{vdw} (which ensures a monochromatic arithmetic progression of arbitrary length). 
Perhaps the best known version of Rado's theorem (often presented in undergraduate courses) is the following, which resolves the case of a single equation.
\begin{thm}[Rado's single equation theorem]\label{thm1}
Let $k \geq 2$ and $a_i \in \mathbb Z\setminus \{0\}$. Then the equation $a_1x_1+a_2x_2+\dots+a_kx_k=0$ has a monochromatic solution in $\mathbb N$ for every finite colouring of $\mathbb N$ if and only if some non-empty subset of the coefficients
$\{a_i : i \in [k] \}$ sum to zero.
\end{thm}

In parallel to progress on Ramsey properties of \emph{random graphs} (see e.g.~\cite{mns, random3}), 
there has been interest in proving random analogues of such results from arithmetic Ramsey theory. 
(This is part of a wider interest in extending classical combinatorial results to the random setting, see e.g.~\cite{conlongowers, schacht} and the survey~\cite{conlonsurvey}.)\COMMENT{AT 27/04: added as I think it is good to put in wider context}
In particular, results of R\"odl and Ruci\'nski~\cite{random4} and Friedgut, R\"odl and Schacht~\cite{frs} together provide a random version of Rado's theorem.

\subsection{A random version of Rado's theorem}
Before we state these results rigorously 
we will introduce some notation and definitions. 
Suppose that $A_1,\dots, A_r$ are integer matrices, and let $S$ be a set of integers. 
If a vector $x=(x_1,\dots,x_k) \in S^k$ satisfies $A_ix=0$  and the $x_i$ are distinct we call $x$ a \emph{$k$-distinct solution} to $A_ix=0$ in $S$. 
We say that $S$ is \emph{$(A_1,\dots, A_r)$-Rado} if given any $r$-colouring of $S$, 
there is some $i \in [r]$ such that there is a $k$-distinct solution $x=(x_1,\dots,x_k)$ to $A_ix=0$ in $S$ so 
that $x_1,\dots ,x_k$ are each coloured with the $i$th colour.
If $A:=A_1=\dots =A_r$ we write \emph{$(A,r)$-Rado} for $(A_1,\dots, A_r)$-Rado.
Similarly, given linear equations $B_1,\dots, B_r$, we define a \emph{$k$-distinct solution of $B_i$} and
\emph{$(B_1,\dots, B_r)$-Rado} analogously.
Note that in the study of random versions of Rado's theorem authors have (implicitly) considered 
the $(A_1,\dots, A_r)$-Rado property, rather than seeking a monochromatic solution that is not necessarily
$k$-distinct (as in the original theorem of Rado). Perhaps a  partial explanation for this can be seen if one considers e.g. the equation
$x+y=2z$; in this case any (monochromatic) set has a solution to this equation (since $w+w=2w$ for any
$w \in \mathbb N$). A more general discussion which further highlights 
 why the literature has  focused on monochromatic $k$-distinct solutions in the random setting is given in Section~\ref{conc}.

A matrix $A$ is \emph{partition regular} if for any finite colouring of $\mathbb{N}$, there is always a monochromatic solution to $Ax=0$. 
As mentioned above, Rado's theorem characterises all those integer matrices $A$ that are partition regular.
A matrix $A$ is \emph{irredundant} if there exists a $k$-distinct solution to $Ax=0$ in $\mathbb{N}$. Otherwise $A$ is \emph{redundant}. 
The study of random versions of Rado's theorem has focused on irredundant partition regular matrices. This is natural since
for every redundant $\ell \times k$ matrix $A$ for which $Ax=0$ has solutions in $\mathbb{N}$,
there exists an irredundant $\ell '\times k'$ matrix $A'$
for some $\ell '< \ell$ and $k' < k$ with the
same family of solutions (viewed as sets). See~\cite[Section 1]{random4} for a full explanation. 
Similarly, we define linear equations to be \emph{irredundant/redundant} analogously.

Index the columns of $A$ by $[k]$. For a partition $W \dot{\cup} \overline{W} = [k]$ of the columns of $A$, we denote by $A_{\overline{W}}$ the matrix obtained from $A$ by restricting to the columns indexed by $\overline{W}$. Let $\rank(A_{\overline{W}})$ be the rank of $A_{\overline{W}}$, where $\rank(A_{\overline{W}})=0$ for $\overline{W}=\emptyset$. We set 
\begin{align}\label{m(A)def}
m(A):=\max_{\substack{W \dot{\cup} \overline{W} = [k] \\ |W|\geq 2}} \frac{|W|-1}{|W|-1+\rank(A_{\overline{W}})-\rank(A)}.
\end{align}
The definition of $m(A)$ was introduced in~\cite{random4}, and as noted there the denominator of $m(A)$ is strictly positive provided that $A$ is irredundant and partition regular.

Suppose now that $A$ is a linear equation with $k$ variables. (We also describe $A$ as having \emph{length} $k$.)
Thus $A$ is of the form $A'x=c$ where $c\in \mathbb Z$
and $A'$ is a $1\times k$ integer matrix (where all terms are non-zero). 
We call $A'$ the \emph{underlying matrix} of $A$.
Note that if $A'$ is irredundant, then so is $A$; this fact is contained within Lemma 4.1 in~\cite{krss}. 
(That is, $A'x=c$ has a $k$-distinct solution in $\mathbb N$ as long as $A'x=0$ does.)
We define $m(A):=m(A')$. 
In this case (provided $k \geq 3$), the value of $m(A)$ is obtained by considering $W=[k]$ and so
\begin{align}\label{1eq}
    m(A)=\frac{k-1}{k-2}.
\end{align}
\smallskip

Recall that $[n]_p$ denotes a set where each element $a \in [n]:=\{1,\dots,n\}$ is included with probability $p$ 
independently of all other elements. R\"odl and Ruci\'nski~\cite{random4} showed that for 
irredundant partition regular matrices $A$, $m(A)$ is an important parameter for determining 
whether $[n]_p$ is $(A,r)$-Rado or not. 

\begin{thm}[R\"odl and Ruci\'nski~\cite{random4}]\label{radores0} 
For all irredundant partition regular full rank matrices $A$ and all positive integers $r\geq 2$, 
there exists a constant $c>0$ such that 
$$
\lim_{n \rightarrow \infty} \mathbb{P}\left[ [n]_p \text{ is } (A,r)\text{-Rado} \right]=0 \quad\text{ if } p < cn^{-1/m(A)}.
$$
\end{thm}

Roughly speaking, Theorem~\ref{radores0} implies that almost all subsets of $[n]$ with 
significantly fewer than $n^{1-1/m(A)}$ elements are not $(A,r)$-Rado
for any irredundant partition regular matrix $A$.
The following theorem  of Friedgut, R\"odl and Schacht~\cite{frs} complements this result, implying that 
almost all subsets of $[n]$ with significantly more than $n^{1-1/m(A)}$ elements are $(A,r)$-Rado
for any irredundant partition regular matrix $A$.

\begin{thm}[Friedgut, R\"odl and Schacht~\cite{frs}]\label{r3} 
For all irredundant partition regular full rank matrices $A$ and all positive integers $r$, there exists a constant $C>0$ such that $$
\lim_{n \rightarrow \infty} \mathbb{P}\left[ [n]_p \text{ is } (A,r)\text{-Rado}\right]=1 \quad\text{ if } p > Cn^{-1/m(A)}.
$$
\end{thm}
So together Theorems~\ref{radores0} and~\ref{r3} show that the threshold 
for the property of being $(A,r)$-Rado is $p=n^{-1/m(A)}$.
Note that earlier Theorem~\ref{r3} was confirmed by Graham, R\"odl and Ruci\'nski~\cite{grr} in the case where $r=2$ and $Ax=0$ corresponds to $x+y=z$, and then by R\"odl and Ruci\'nski~\cite{random4} in the case when $A$ is so-called density regular.
Since its proof, generalised versions of Theorem~\ref{r3} have been obtained via applications of the container method~\cite{hst, sp}.
A \emph{sharp threshold} version of van der Waerden's theorem for random subsets
of $\mathbb Z_n$ has also been obtained~\cite{sharp}.\COMMENT{AT 27/4: added}

Whilst preparing this paper, we discovered a bug in the original proof of Theorem~\ref{radores0} (this is explained further in Section~\ref{sec:mainproof2}).
Thus, an aim of this paper is to give a proof of Theorem~\ref{radores0}. In fact,
 we prove a more general result; see Theorem~\ref{mainthm2}.

\subsection{An asymmetric version of the random Rado theorem}
As noted  e.g. in~\cite{AP}, one can deduce an  {\emph{asymmetric}} version of Rado's theorem from the original (symmetric) result~\cite{rado}.
In particular, if $A_1,\dots,A_r$ are partition regular matrices 
then $\mathbb N$ is $(A_1,\dots,A_r)$-Rado.
(Note though that even a weak version of the converse statement is not true. For example,
there are $2$-colourings of $\mathbb N$ without a monochromatic solution to $x=2y$, and also
such $2$-colourings of $\mathbb N$ for $x=4y$. On the other hand, however one $2$-colours $\{1,2,4,8,16\}$, one obtains a red solution to $x=2y$ or blue solution to $x=4y$.)

It is also natural to seek an asymmetric version of the random Rado theorem. This question was first considered
by the authors and Staden~\cite{hst}  who proved the following: given any $r\geq 2$ and any irredundant full
rank partition regular matrices $A_1,\dots, A_r$ with $m(A_1)\geq \dots \geq m(A_r)$, there is a constant $C>0$
so that $\lim_{n \rightarrow \infty} \mathbb{P}\left[ [n]_p \text{ is } (A_1,\dots,A_r)\text{-Rado}\right]=1$ { if } $p > Cn^{-1/m(A_1)}$.

In general the bound on $p$ in this result is not believed to be best possible (unless $m(A_1)=m(A_2)$).
Indeed, recently Aigner-Horev and Person~\cite{AP} have given a conjecture on the threshold for
the asymmetric Rado property. To state this conjecture, we need one more definition.
Let $A$ and $B$ be two integer matrices,  where $A$ is an $\ell _A \times k_A$ matrix and 
$B$ is an $\ell _B \times k_B$ matrix.
Then define
\begin{align}\label{eq:mAB}
m(A,B):= \max_{\stackrel{W \dot\cup \overline{W} = [k_A]}{|W| \geq 2}} \frac{|W|}{|W|-1+\rank(A_{\overline{W}})-\rank(A)+1/m(B)}.
\end{align}
As observed in~\cite[Observation 4.13]{AP}, if 
$A$ and $B$ are partition regular and irredundant and
$m(A)\geq m(B)$, then $m(A,B)\geq m(B)$ and 
$m(A,A)=m(A)$.
If $A$ and $B$ are linear equations each of length at least three then we define
$m(A,B)$ in an analogous way (i.e. $m(A,B):= m(A',B')$ where $A',B'$ are the underlying matrices of
$A$ and $B$ respectively). See~\cite[Page 4]{AP}
for an intuitive explanation of the parameter $m(A,B)$.

\begin{conjecture}[Aigner-Horev and Person~\cite{AP}]\label{conj}
Let $A_1,\dots,A_r$ be $r$ irredundant partition regular matrices of full rank where
$m(A_1)\geq m(A_2)\geq \dots \geq m(A_r)$. Then there exists $0<c<C$ such that the following holds
\begin{align*}
\lim_{n \to \infty} \mathbb{P}[[n]_p \text{ is $(A_1,\dots, A_r)$-Rado}]=
\begin{cases}
1 \text{ if } p>Cn^{-1/m(A_1,A_2)}; \\
0 \text{ if } p<cn^{-1/m(A_1,A_2)}.
\end{cases}
\end{align*}
\end{conjecture}
In particular, if true, Conjecture~\ref{conj} provides a wide generalisation of the (symmetric)
random Rado theorem (Theorems~\ref{radores0} and~\ref{r3}).
Note that the reader might recognise  parallels between this conjecture and the \emph{Kohayakawa--Kreuter conjecture} for asymmetric Ramsey properties of random graphs; see~\cite{AP} for more details.
In~\cite{AP}, Aigner-Horev and Person proved the $1$-statement ($p>Cn^{-1/m(A_1,A_2)}$) of Conjecture~\ref{conj} via the container method.
Thus, only the $0$-statement ($p<cn^{-1/m(A_1,A_2)}$) now remains open.

In this paper, we make significant progress on this problem, including resolving the conjecture in the case that
each of the $A_i$s corresponds to  linear equations (rather than  systems of linear equations).
Note that  such irredundant partition regular linear equations have at least three variables
by Theorem~\ref{thm1}. 
In fact,  we prove the following more general result.

\begin{thm}\label{mainthm}
Let $k_B \geq k_A \geq 3$ be positive integers. Then there exists a constant $c>0$ such that the following holds.
Let $A$ and $B$ be
linear equations of lengths $k_A$ and $k_B$ respectively so that their underlying matrices
are both irredundant.\footnote{Recall from the discussion before (\ref{1eq}) this latter condition implies the linear equations $A$ and $B$ are also irredundant.}
If
$$p \leq cn^{-\frac{k_Ak_B-k_A-k_B}{k_Ak_B-k_A}}$$ then 
$
\lim_{n \to \infty} \mathbb{P}[ [n]_p \text{ is  $(A,B)$-Rado}]=0.
$
\end{thm}
As we now show, Theorem~\ref{mainthm} easily implies the $0$-statement of Conjecture~\ref{conj} for linear equations.
\begin{cor}
Let $A_1,\dots,A_r$ be irredundant homogeneous  partition regular linear equations, each on at least $3$ variables, where
$m(A_1)\geq m(A_2)\geq \dots \geq m(A_r)$. Then there exists $c>0$ such that the following holds:
\begin{align*}
\lim_{n \to \infty} \mathbb{P}[[n]_p \text{ is $(A_1,\dots, A_r)$-Rado}]=
0 \text{ if } p<cn^{-1/m(A_1,A_2)}.
\end{align*}
\end{cor}
\proof
Write $A:=A_1$ and $B:=A_2$, so that $A$ and $B$ are linear equations of lengths $k_A$ and $k_B$ respectively.
As $m(A)\geq m(B)$ we have that $k_B\geq k_A\geq 3$.
Further, by definition and (\ref{1eq}),
$$m(A,B)=\frac{k_Ak_B-k_A}{k_Ak_B-k_A-k_B}.$$
Indeed, the term in (\ref{eq:mAB}) is maximised when $W=[k_A]$.
So if $p\leq c n^{-1/m(A,B)}$ then Theorem~\ref{mainthm} implies that
$\lim_{n \to \infty} \mathbb{P}[ [n]_p \text{ is  $(A,B)$-Rado}]=0.$
This immediately implies that \\
$\lim_{n \to \infty} \mathbb{P}[[n]_p \text{ is $(A_1,\dots, A_r)$-Rado}]=
0.$
\endproof
Note that Theorem~\ref{mainthm} allows for $A$ and $B$ to be inhomogeneous equations (i.e. $a_1x_1+\dots+a_kx_k=b$, $b \not=0$). It also allows us to consider linear equations  that are not partition regular. For example, if $A$ is $2x+2y=z$, then it is  not partition regular, however, $\mathbb N$ is $(A,2)$-Rado (as observed in~\cite{rado}); 
so it is natural to seek random Rado-type results for such equations also.
Furthermore, as we now explain, when one of the linear equations or its underlying matrix is redundant, the random Rado problem is trivial:
\begin{itemize}
\item Consider linear equations $A_1,\dots, A_r$. If for some $i \in [r]$, $A_i$ is redundant then even 
$\mathbb N$ is not $(A_1,\dots, A_r)$-Rado; indeed, colour every element of $\mathbb N$ with the $i$th 
colour. Thus, the random Rado problem in this case is trivial.

\item Suppose that each of $A_1,\dots, A_r$ is irredundant with length at least $3$, but for some $i \in [r]$, the underlying matrix
$A'_i$ of $A_i$ is redundant.
Then $A_i$ corresponds to $a_1 x_1 + \dots + a_k x_k=b$ 
where $a_i, b$ are all positive integers.\footnote{First observe that $A_i$ corresponds to $a_1 x_1+ \cdots+ a_k x_k =b$ where $a_i,b \in \mathbb{Z} \setminus \{0\}$ and $k \geq 3$. If at least one $a_i$ is positive and at least one $a_j$ is negative, then $a_1 x_1+ \cdots+ a_k x_k =0$ would have a solution in $\mathbb{N}$, which is not true since $A'_i$ is redundant. Since $A_i$ is irredundant we therefore obtain that $a_i,b$ are all positive or all negative, and without loss of generality we may assume all positive.}
In this case there are a finite number of solutions to $A_i$ in $\mathbb{N}$.
Thus, if $p=o(1)$ then with high probability (w.h.p.) no solutions to $A_i$ will be present in $[n]_p$; again this immediately
implies that w.h.p. $[n]_p$ is not $(A_1,\dots, A_r)$-Rado. 
Note that this class includes linear equations $A_1,\dots,A_r$ such that $\mathbb{N}$ is $(A_1,\dots,A_r)$-Rado. 
For example, let $r=2$, $A_1$ be $x+y=2z$ and $A_2$ be $x+y+z=C$ where $C$ is a sufficiently large constant.\footnote{In particular, $[C]$ is $(A_1,A_2)$-Rado: to avoid a red solution to $A_1$ only $o(C)$ colours may be coloured red, 
so most numbers in $[C]$ are blue, and a blue solution to $A_2$ can be found.}

\end{itemize}
It would  be interesting to deduce a matching $1$-statement for linear equations covered by Theorem~\ref{mainthm} but not by Conjecture~\ref{conj}; see Section~\ref{conc} for further discussion on this.

\smallskip

As mentioned earlier, we give a proof of a generalisation of Theorem~\ref{radores0}. Before we can state this result we need some more notation.

Define an $\ell \times k$ matrix $A$ of full rank to be \emph{strictly balanced} if, for every $W \subseteq [k]$, $2 \leq |W| <k$, the following inequality holds:
\begin{align}\label{eq:strictlyb}
\frac{|W|-1}{|W|-1+\rank(A_{\overline{W}})-\ell} < \frac{k-1}{k-1-\ell}.
\end{align}
Thus, if $A$ is strictly balanced then $m(A)=(k-1)/(k-1-\ell)$.
Given an irredundant partition regular matrix $A$, a \emph{core} $C(A)$ is a matrix obtained from $A$ by deleting rows and columns of $A$,
such that $m(C(A))=m(A)$, which is irredundant, of full rank and is strictly balanced. A core always exists by Lemma 7.1 of~\cite{random4}.

Given an inhomogeneous system of linear equations $Ax=b$, we call $A$ the \emph{underlying matrix} of the system.
We define such a system to be irredundant/redundant analogously to linear equations. 
Note again by Lemma 4.1 in~\cite{krss} that if $A$ is irredundant, then so is $Ax=b$.
Given a system of linear equations $B$, we
 write $C(B)$ to denote a core of the underlying matrix of $B$.

\begin{thm}\label{mainthm2}
Let $k, \ell$ be positive integers such that $k \geq \ell+2$.
Then there exists a constant $c>0$ such that the following holds.
Let $A$ and $B$ be systems of linear equations for which both of their underlying matrices are irredundant, partition regular, and of full rank,
and their cores $C(A)$ and $C(B)$ are both of dimension $\ell \times k$. 
If $$p \leq cn^{-1/m(A,B)}=cn^{-\frac{k-\ell-1}{k-1}}$$ then  
$
\lim_{n \to \infty} \mathbb{P}[ [n]_p \text{ is  $(A,B)$-Rado}]=0.
$
\end{thm}



It is easy to see that this theorem implies 
Theorem~\ref{radores0}; in particular, if $A$ is an $\ell \times k$ irredundant  partition regular matrix of full rank, then $k \geq \ell +2$ (see e.g. 
\cite[Proposition 4.3]{hst}).
Notice that Theorem~\ref{mainthm2} 
also resolves Conjecture~\ref{conj} in the case when $A_1$ and $A_2$ are strictly balanced and have the same dimensions (note $m(A_1)=m(A_2)$ in this case).
Theorem~\ref{mainthm2} as stated does not quite imply Theorem~\ref{mainthm} in the
case when $k_A=k_B$ (as Theorem~\ref{mainthm2} assumes that the underlying matrices of $A$ and $B$ are partition regular). So we in fact prove an even more general (but technical) version of Theorem~\ref{mainthm2} that contains the $k_A=k_B$ case of Theorem~\ref{mainthm};
see Theorem~\ref{mainthm2general} in Section~\ref{sec:mainproof2}.

In Section~\ref{sec:mainproof} we prove Theorem~\ref{mainthm} in the case where $k_B>k_A$. In Section~\ref{31} we outline the approach of the proof of Theorem~\ref{radores0} in~\cite{random4}. In Section~\ref{32} we state 
Theorem~\ref{mainthm2general} which as just described is a generalisation of Theorem~\ref{mainthm2} that also implies the $k_A=k_B$ case of Theorem~\ref{mainthm}.  Theorem~\ref{mainthm2general} is then proved in  Section~\ref{33}. In Section~\ref{conc} we conclude the paper with some open problems.

{\bf Remark:} In  work simultaneous to our own, Zohar~\cite{zohar} has given a proof of
Conjecture~\ref{conj} in the case when each $A_i$ corresponds to an arithmetic progression.


\subsection{Notation}\label{sec:notation}
As in the proof of Theorem~\ref{radores0} in~\cite{random4}, we prove our two main results by considering an
auxiliary hypergraph.
For a (hyper)graph $H$, we define $V(H)$ and $E(H)$ to be the vertex and edge sets of $H$ respectively. 
For a set $A \subseteq V(H)$, we define $H[A]$ to be the induced subgraph of $H$ on the vertex set $A$. 
For an edge set $X \subseteq E(H)$, we define $H - X$ to be hypergraph with vertex set $V(H)$ and edge set $E(H) \setminus X$. 
We use the convention that the set of natural numbers $\mathbb N$ does not include zero.



\section{Proof of Theorem~\ref{mainthm}: the $k_B>k_A$ case}\label{sec:mainproof}

Suppose that $A$ and  $B$ are linear equations as in the statement of the theorem with lengths $k_A$ and $k_B$ respectively where
 $k_B > k_A \geq 3$.

Let $c>0$ be a constant sufficiently small compared to $1/k_A$ and $1/k_B$. 
(So the choice of $c$ depends on $k_A$ and $k_B$ only, and not the particular linear equations $A$ and $B$.)
It suffices to prove the theorem in the case when 
$p=cn^{-\frac{k_Ak_B-k_A-k_B}{k_Ak_B-k_A}}$. 

Consider the  \emph{associated hypergraph} $G=G(n,p,A,B)$: here $V(G):=[n]_p$ and the  edge set of $G$ 
consists of \emph{$A$-edges} which are 
edges of size $k_A$ that precisely correspond to the 
$k_A$-distinct solutions of $A$ in $[n]_p$, and \emph{$B$-edges} which are 
edges of size $k_B$ that precisely correspond to the 
$k_B$-distinct solutions of $B$ in $[n]_p$. 

Our aim is to show that w.h.p. there is a red-blue colouring of the vertices of $G$ so that 
there are no red $A$-edges and no blue $B$-edges.
In particular, call $G$ \emph{Rado} if it has the property that however its vertices are red-blue coloured, 
there is always a red $A$-edge or a blue $B$-edge; call 
a spanning 
subgraph $H$ of $G$ \emph{Rado minimal} if $H$ is Rado however it is no longer Rado under the deletion of any edge. 
(Such a definition makes sense since being Rado is a monotone hypergraph property.) 
If $G$ is Rado, fix a Rado minimal subgraph $H$ of $G$. Otherwise set $H:=\emptyset$. 
So it suffices to prove that w.h.p. $H=\emptyset$. 

The first claim is a generalisation of a statement (Proposition~7.4 in~\cite{random4}) used in the proof of
Theorem~\ref{radores0}. The proof follows in the same manner.
\begin{claim}\label{cl:determ}
Suppose $H$ is non-empty (i.e., $H$ is Rado minimal). Then for every $A$-edge $a$ of $H$ and every vertex $v \in a$, there exists a $B$-edge $b$ such that $a \cap b =v$. 
Similarly, for every $B$-edge $b$ of $H$ and every vertex $v \in b$, there exists an $A$-edge $a$ such that $a \cap b =v$.
\end{claim}

\begin{proof}
Let $a$ be an $A$-edge, and let $v \in a$ be such that for all $B$-edges $b$ such that $v \in b$, 
there exists another vertex $w \in a$ such that $w \in b$. Since $H$ is Rado minimal, 
it is possible to red-blue colour $H-a$ so that 
there are no red $A$-edges or blue $B$-edges. Thus once we add $a$ back, it must be the case that 
$a$ is red since $H$ is Rado. But then change the colour of $v$ to blue. 
If there is a red $A$-edge or blue $B$-edge now, it must be a blue $B$-edge which contains $v$. 
However all $B$-edges containing $v$ also contain another vertex from $a$ which is red, thus we obtain a contradiction. 
The second statement follows by a symmetrical argument. 
\end{proof}

\subsection{Notation}
We start by defining some hypergraph notation. Given an edge order $e_1, \dots, e_t$ of the edges of a hypergraph, 
we call a vertex $v$ \emph{new in $e_i$} if $v \in e_i$ but $v \not \in e_j$ for all $j<i$.
Otherwise we call $v \in e_i$ \emph{old in $e_i$}. 
Clearly each vertex is new in one edge, and old in any subsequent edge that it appears in. 
We call an edge order \emph{valid} if there is at least one new vertex in every edge. This notion is crucial for our proof. Indeed, if one can show a hypergraph $F$ has a valid edge order then (via Claim~\ref{cl:copiesofS} below) we can obtain a good upper bound on the expected number of copies of $F$ in $G$.
\COMMENT{AT added 25 april, as I wanted to emphasise to reader its a definition they must remember!}

Further, define the following hypergraphs.
\begin{enumerate}
\item[(A1)] An \emph{$A$-path of length $s$} (for $s \in \mathbb N$) consists of
 a set of $s$ $A$-edges $a_1, \dots, a_s$ where $|a_i \cap a_j| = 1$ for $i<j$ if $j-i=1$, and $0$ otherwise. 
\item[(A2)] An \emph{$A$-cycle of length $s$} (for $s\geq 3$)
consists of a set of $s$ $A$-edges $a_1, \dots, a_s$ where given any $i<j$, $|a_i \cap a_j| = 1$   
if (i) $j-i=1$ or (ii) $(i,j)=(1,s)$, and $|a_i \cap a_j| = 0$ otherwise. 
\item[(A3)] An \emph{$A$-tree} spans a set of $s$ $A$-edges such that there exists an edge order $a_1, \dots, a_s$ 
where for each $2 \leq i \leq s$, $a_i$ has precisely one old vertex. 

\item[(AB0)] An \emph{$AB$-set} consists of a $B$-edge $b$ with vertices $v_1,\dots,v_{k_B}$ and 
a set of pairwise disjoint $A$-edges $a_1,\dots,a_{k_B}$ with $a_i \cap b = v_i$ for each $i \in [k_B]$. 

\item[(AB1)] An \emph{$AB$-path of length $t$} (for $t \in \mathbb N$) consists of a collection of 
pairwise disjoint $B$-edges $b_i$ (for $i \in [t]$) and 
a collection of pairwise disjoint $A$-edges $a_{j}$ (for $j \in [t(k_B-1)+1]$) such that 
$b_i$ together with $a_{(i-1)(k_B-1)+1},\dots  ,a_{i(k_B-1)+1}$ 
forms an $AB$-set (for each $i \in [t]$).
\item[(AB2)] An \emph{$AB$-cycle of length $t$} (for $t \geq 2$) consists of a collection of 
$B$-edges $b_i$ (for each $i \in [t]$) and 
a collection of pairwise disjoint $A$-edges $a_{j}$ (for $j \in [t(k_B-1)]$) such that: 
\begin{itemize}
\item given any $i \in [t-1]$, $b_i$ together with $a_{(i-1)(k_B-1)+1},\dots  ,a_{i(k_B-1)+1}$ forms an $AB$-set;
\item $b_t$ together with $a_{(t-1)(k_B-1)+1},\dots ,a_{t(k_B-1)}$ and $a_1$ forms an $AB$-set.
\item $|b_i \cap b_j|=0$ for all $i<j$, except for $(i,j)=(1,t)$, where we have either $|b_1 \cap b_t|=0$ or $|b_1 \cap b_t|=1$.
\end{itemize}

\item[(AB3)] An \emph{$AB$-cycle-path with parameters $s,t$} ($s \not = 1$ and $t$ 
are non-negative integers where $(s,t)\not =(0,0)$), is the following structure:
If $s=0$ then it is an $AB$-path of length $t$.
If $t=0$ then it is an $AB$-cycle of length $s$.
If $s \geq 2$ and $t \geq 1$, then it is an $AB$-cycle $S$ of length $s$ together with an $AB$-path $T$ of length $t$,
where, letting $b_1,\dots,b_t$ and $a_1$ be as in the definition (AB1) of $T$, 
we have $V(S) \cap V(T)=a_1$, $E(S)\cap E(T)=\{a_1\}$ and the vertex $v:=a_1\cap b_1$
does not lie in any $B$-edge in $S$.
\end{enumerate}
Note that for all of the above hypergraphs, one can easily derive a valid edge order. 
Further, note that $AB$-paths, $A$-paths and $A$-cycles of a fixed size and $AB$-sets are unique,
whereas there are two different $AB$-cycles of a fixed size; 
one where the first and final $B$-edges intersect, and one where they do not.
See Figures~\ref{figAB0}--\ref{figAB3} for examples of (AB0)--(AB3) respectively;
in all of these pictures the $B$-edges are shaded in grey to help emphasise which edges are $A$-edges and which are $B$-edges (though recall that the edges do not have colours).

\begin{figure}[hb]
\begin{center}
\begin{tikzpicture}[scale=0.90]
\node[vertex] at (0,0) (1) {};
\node[vertex] at (0,1) (2) {};
\node[vertex] at (0,2) (3) {};
\node[vertex] at (0,3) (4) {};
\node[vertex] at (1,0) (5) {};
\node[vertex] at (1,1) (6) {};
\node[vertex] at (1,2) (7) {};
\node[vertex] at (1,3) (8) {};
\node[vertex] at (2,0) (9) {};
\node[vertex] at (2,1) (10) {};
\node[vertex] at (2,2) (11) {};
\node[vertex] at (2,3) (12) {};
\node[vertex] at (3,0) (13) {};
\node[vertex] at (3,1) (14) {};
\node[vertex] at (3,2) (15) {};
\node[vertex] at (3,3) (16) {};
\node[vertex] at (4,0) (17) {};
\node[vertex] at (4,1) (18) {};
\node[vertex] at (4,2) (19) {};
\node[vertex] at (4,3) (20) {};

\filldraw[fill opacity=0.4,fill=gray!70] 
($(1) + (0,0.34)$) to[out=180,in=90] ($(1) + (-0.34,0)$)  to[out=270,in=180] ($(1) + (0,-0.34)$)  to ($(17) + (0,-0.34)$) to[out=0,in=270] ($(17) + (0.34,0)$) to[out=90,in=0] ($(17) + (0,0.34)$) to ($(1) + (0,0.34)$);

\filldraw[fill opacity=0,fill=white!70] 
($(4) + (0.4,0)$) to[out=90,in=0] ($(4) + (0,0.4)$)  to[out=180,in=90] ($(4) + (-0.4,0)$)  to ($(1) + (-0.4,0)$) to[out=270,in=180] ($(1) + (0,-0.4)$) to[out=0,in=270] ($(1) + (0.4,0)$) to ($(4) + (0.4,0)$);

\filldraw[fill opacity=0,fill=white!70] 
($(8) + (0.4,0)$) to[out=90,in=0] ($(8) + (0,0.4)$)  to[out=180,in=90] ($(8) + (-0.4,0)$)  to ($(5) + (-0.4,0)$) to[out=270,in=180] ($(5) + (0,-0.4)$) to[out=0,in=270] ($(5) + (0.4,0)$) to ($(8) + (0.4,0)$);

\filldraw[fill opacity=0,fill=white!70] 
($(12) + (0.4,0)$) to[out=90,in=0] ($(12) + (0,0.4)$)  to[out=180,in=90] ($(12) + (-0.4,0)$)  to ($(9) + (-0.4,0)$) to[out=270,in=180] ($(9) + (0,-0.4)$) to[out=0,in=270] ($(9) + (0.4,0)$) to ($(12) + (0.4,0)$);

\filldraw[fill opacity=0,fill=white!70] 
($(16) + (0.4,0)$) to[out=90,in=0] ($(16) + (0,0.4)$)  to[out=180,in=90] ($(16) + (-0.4,0)$)  to ($(13) + (-0.4,0)$) to[out=270,in=180] ($(13) + (0,-0.4)$) to[out=0,in=270] ($(13) + (0.4,0)$) to ($(16) + (0.4,0)$);

\filldraw[fill opacity=0,fill=white!70] 
($(20) + (0.4,0)$) to[out=90,in=0] ($(20) + (0,0.4)$)  to[out=180,in=90] ($(20) + (-0.4,0)$)  to ($(17) + (-0.4,0)$) to[out=270,in=180] ($(17) + (0,-0.4)$) to[out=0,in=270] ($(17) + (0.4,0)$) to ($(20) + (0.4,0)$);

\node at (4.7,0) {$b$};
\node at (0,3.7) {$a_1$};
\node at (1,3.7) {$a_2$};
\node at (2,3.7) {$a_3$};
\node at (3,3.7) {$a_4$};
\node at (4,3.7) {$a_5$};
\end{tikzpicture}
\end{center}
\caption{An example of an $AB$-set with $k_A=4$ and $k_B=5$. The $B$-edge is shaded in grey.}
\label{figAB0}
%
%
\begin{center}
\begin{tikzpicture}[scale=0.90]
\node[vertex] at (0,0) (1) {};
\node[vertex] at (0,1) (2) {};
\node[vertex] at (0,2) (3) {};
\node[vertex] at (0,3) (4) {};
\node[vertex] at (1,0) (5) {};
\node[vertex] at (1,1) (6) {};
\node[vertex] at (1,2) (7) {};
\node[vertex] at (1,3) (8) {};
\node[vertex] at (2,0) (9) {};
\node[vertex] at (2,1) (10) {};
\node[vertex] at (2,2) (11) {};
\node[vertex] at (2,3) (12) {};
\node[vertex] at (3,0) (13) {};
\node[vertex] at (3,1) (14) {};
\node[vertex] at (3,2) (15) {};
\node[vertex] at (3,3) (16) {};
\node[vertex] at (4,0) (17) {};
\node[vertex] at (4,1) (18) {};
\node[vertex] at (4,2) (19) {};
\node[vertex] at (4,3) (20) {};

\node[vertex] at (5,0) (21) {};
\node[vertex] at (5,1) (22) {};
\node[vertex] at (5,2) (23) {};
\node[vertex] at (5,3) (24) {};
\node[vertex] at (6,0) (25) {};
\node[vertex] at (6,1) (26) {};
\node[vertex] at (6,2) (27) {};
\node[vertex] at (6,3) (28) {};
\node[vertex] at (7,0) (29) {};
\node[vertex] at (7,1) (30) {};
\node[vertex] at (7,2) (31) {};
\node[vertex] at (7,3) (32) {};
\node[vertex] at (8,0) (33) {};
\node[vertex] at (8,1) (34) {};
\node[vertex] at (8,2) (35) {};
\node[vertex] at (8,3) (36) {};

\node[vertex] at (9,0) (37) {};
\node[vertex] at (9,1) (38) {};
\node[vertex] at (9,2) (39) {};
\node[vertex] at (9,3) (40) {};
\node[vertex] at (10,0) (41) {};
\node[vertex] at (10,1) (42) {};
\node[vertex] at (10,2) (43) {};
\node[vertex] at (10,3) (44) {};
\node[vertex] at (11,0) (45) {};
\node[vertex] at (11,1) (46) {};
\node[vertex] at (11,2) (47) {};
\node[vertex] at (11,3) (48) {};
\node[vertex] at (12,0) (49) {};
\node[vertex] at (12,1) (50) {};
\node[vertex] at (12,2) (51) {};
\node[vertex] at (12,3) (52) {};

\filldraw[fill opacity=0.4,fill=gray!70] 
($(1) + (0,0.34)$) to[out=180,in=90] ($(1) + (-0.34,0)$)  to[out=270,in=180] ($(1) + (0,-0.34)$)  to ($(17) + (0,-0.34)$) to[out=0,in=270] ($(17) + (0.34,0)$) to[out=90,in=0] ($(17) + (0,0.34)$) to ($(1) + (0,0.34)$);

\filldraw[fill opacity=0.4,fill=gray!70] 
($(20) + (0,0.34)$) to[out=180,in=90] ($(20) + (-0.34,0)$)  to[out=270,in=180] ($(20) + (0,-0.34)$)  to ($(36) + (0,-0.34)$) to[out=0,in=270] ($(36) + (0.34,0)$) to[out=90,in=0] ($(36) + (0,0.34)$) to ($(20) + (0,0.34)$);

\filldraw[fill opacity=0.4,fill=gray!70] 
($(33) + (0,0.34)$) to[out=180,in=90] ($(33) + (-0.34,0)$)  to[out=270,in=180] ($(33) + (0,-0.34)$)  to ($(49) + (0,-0.34)$) to[out=0,in=270] ($(49) + (0.34,0)$) to[out=90,in=0] ($(49) + (0,0.34)$) to ($(33) + (0,0.34)$);

\filldraw[fill opacity=0,fill=white!70] 
($(4) + (0.4,0)$) to[out=90,in=0] ($(4) + (0,0.4)$)  to[out=180,in=90] ($(4) + (-0.4,0)$)  to ($(1) + (-0.4,0)$) to[out=270,in=180] ($(1) + (0,-0.4)$) to[out=0,in=270] ($(1) + (0.4,0)$) to ($(4) + (0.4,0)$);

\filldraw[fill opacity=0,fill=white!70] 
($(8) + (0.4,0)$) to[out=90,in=0] ($(8) + (0,0.4)$)  to[out=180,in=90] ($(8) + (-0.4,0)$)  to ($(5) + (-0.4,0)$) to[out=270,in=180] ($(5) + (0,-0.4)$) to[out=0,in=270] ($(5) + (0.4,0)$) to ($(8) + (0.4,0)$);

\filldraw[fill opacity=0,fill=white!70] 
($(12) + (0.4,0)$) to[out=90,in=0] ($(12) + (0,0.4)$)  to[out=180,in=90] ($(12) + (-0.4,0)$)  to ($(9) + (-0.4,0)$) to[out=270,in=180] ($(9) + (0,-0.4)$) to[out=0,in=270] ($(9) + (0.4,0)$) to ($(12) + (0.4,0)$);

\filldraw[fill opacity=0,fill=white!70] 
($(16) + (0.4,0)$) to[out=90,in=0] ($(16) + (0,0.4)$)  to[out=180,in=90] ($(16) + (-0.4,0)$)  to ($(13) + (-0.4,0)$) to[out=270,in=180] ($(13) + (0,-0.4)$) to[out=0,in=270] ($(13) + (0.4,0)$) to ($(16) + (0.4,0)$);

\filldraw[fill opacity=0,fill=white!70] 
($(20) + (0.4,0)$) to[out=90,in=0] ($(20) + (0,0.4)$)  to[out=180,in=90] ($(20) + (-0.4,0)$)  to ($(17) + (-0.4,0)$) to[out=270,in=180] ($(17) + (0,-0.4)$) to[out=0,in=270] ($(17) + (0.4,0)$) to ($(20) + (0.4,0)$);

\filldraw[fill opacity=0,fill=white!70] 
($(24) + (0.4,0)$) to[out=90,in=0] ($(24) + (0,0.4)$)  to[out=180,in=90] ($(24) + (-0.4,0)$)  to ($(21) + (-0.4,0)$) to[out=270,in=180] ($(21) + (0,-0.4)$) to[out=0,in=270] ($(21) + (0.4,0)$) to ($(24) + (0.4,0)$);

\filldraw[fill opacity=0,fill=white!70] 
($(28) + (0.4,0)$) to[out=90,in=0] ($(28) + (0,0.4)$)  to[out=180,in=90] ($(28) + (-0.4,0)$)  to ($(25) + (-0.4,0)$) to[out=270,in=180] ($(25) + (0,-0.4)$) to[out=0,in=270] ($(25) + (0.4,0)$) to ($(28) + (0.4,0)$);

\filldraw[fill opacity=0,fill=white!70] 
($(32) + (0.4,0)$) to[out=90,in=0] ($(32) + (0,0.4)$)  to[out=180,in=90] ($(32) + (-0.4,0)$)  to ($(29) + (-0.4,0)$) to[out=270,in=180] ($(29) + (0,-0.4)$) to[out=0,in=270] ($(29) + (0.4,0)$) to ($(32) + (0.4,0)$);

\filldraw[fill opacity=0,fill=white!70] 
($(36) + (0.4,0)$) to[out=90,in=0] ($(36) + (0,0.4)$)  to[out=180,in=90] ($(36) + (-0.4,0)$)  to ($(33) + (-0.4,0)$) to[out=270,in=180] ($(33) + (0,-0.4)$) to[out=0,in=270] ($(33) + (0.4,0)$) to ($(36) + (0.4,0)$);

\filldraw[fill opacity=0,fill=white!70] 
($(40) + (0.4,0)$) to[out=90,in=0] ($(40) + (0,0.4)$)  to[out=180,in=90] ($(40) + (-0.4,0)$)  to ($(37) + (-0.4,0)$) to[out=270,in=180] ($(37) + (0,-0.4)$) to[out=0,in=270] ($(37) + (0.4,0)$) to ($(40) + (0.4,0)$);

\filldraw[fill opacity=0,fill=white!70] 
($(44) + (0.4,0)$) to[out=90,in=0] ($(44) + (0,0.4)$)  to[out=180,in=90] ($(44) + (-0.4,0)$)  to ($(41) + (-0.4,0)$) to[out=270,in=180] ($(41) + (0,-0.4)$) to[out=0,in=270] ($(41) + (0.4,0)$) to ($(44) + (0.4,0)$);

\filldraw[fill opacity=0,fill=white!70] 
($(48) + (0.4,0)$) to[out=90,in=0] ($(48) + (0,0.4)$)  to[out=180,in=90] ($(48) + (-0.4,0)$)  to ($(45) + (-0.4,0)$) to[out=270,in=180] ($(45) + (0,-0.4)$) to[out=0,in=270] ($(45) + (0.4,0)$) to ($(48) + (0.4,0)$);

\filldraw[fill opacity=0,fill=white!70] 
($(52) + (0.4,0)$) to[out=90,in=0] ($(52) + (0,0.4)$)  to[out=180,in=90] ($(52) + (-0.4,0)$)  to ($(49) + (-0.4,0)$) to[out=270,in=180] ($(49) + (0,-0.4)$) to[out=0,in=270] ($(49) + (0.4,0)$) to ($(52) + (0.4,0)$);

\node at (2.5,-0.7) {$b_1$};
\node at (6.5,3.7) {$b_2$};
\node at (10.5,-0.7) {$b_3$};

\node at (0,3.7) {$a_1$};
\node at (1,3.7) {$a_2$};
\node at (2,3.7) {$a_3$};
\node at (3,3.7) {$a_4$};
\node at (4,3.7) {$a_5$};
\node at (5,-0.7) {$a_6$};
\node at (6,-0.7) {$a_7$};
\node at (7,-0.7) {$a_8$};
\node at (8,-0.7) {$a_9$};
\node at (9,3.7) {$a_{10}$};
\node at (10,3.7) {$a_{11}$};
\node at (11,3.7) {$a_{12}$};
\node at (12,3.7) {$a_{13}$};
\end{tikzpicture}
\end{center}
\caption{An example of an $AB$-path of length $3$ with $k_A=4$ and $k_B=5$. The $B$-edges $b_1,b_2,b_3$ are shaded in grey. The $AB$-sets are $\{b_1, a_1, a_2, a_3, a_4, a_5\}$, $\{b_2, a_5, a_6, a_7, a_8, a_9\}$ and $\{b_3, a_9, a_{10}, a_{11}, a_{12},a_{13}\}$.}
\label{figAB1}
\end{figure}
\begin{figure}[ht]
\begin{center}
\begin{tikzpicture}[scale=0.90]
\node[vertex] at (2,2) (1) {};
\node[vertex] at (3,2) (2) {};
\node[vertex] at (4,2) (3) {};
\node[vertex] at (5,2) (4) {};
\node[vertex] at (1,1) (5) {};
\node[vertex] at (0,0) (6) {};
\node[vertex] at (3,1) (7) {};
\node[vertex] at (3,0) (8) {};
\node[vertex] at (4,1) (9) {};
\node[vertex] at (4,0) (10) {};
\node[vertex] at (6,1) (11) {};
\node[vertex] at (7,0) (12) {};
\node[vertex] at (5,3) (13) {};
\node[vertex] at (5,4) (14) {};
\node[vertex] at (5,5) (15) {};
\node[vertex] at (6,3) (16) {};
\node[vertex] at (7,3) (17) {};
\node[vertex] at (6,4) (18) {};
\node[vertex] at (7,4) (19) {};
\node[vertex] at (6,6) (20) {};
\node[vertex] at (7,7) (21) {};
\node[vertex] at (4,5) (22) {};
\node[vertex] at (3,5) (23) {};
\node[vertex] at (2,5) (24) {};
\node[vertex] at (4,6) (25) {};
\node[vertex] at (4,7) (26) {};
\node[vertex] at (3,6) (27) {};
\node[vertex] at (3,7) (28) {};
\node[vertex] at (1,6) (29) {};
\node[vertex] at (0,7) (30) {};
\node[vertex] at (2,4) (31) {};
\node[vertex] at (2,3) (32) {};
\node[vertex] at (1,4) (33) {};
\node[vertex] at (0,4) (34) {};
\node[vertex] at (1,3) (35) {};
\node[vertex] at (0,3) (36) {};

\filldraw[fill opacity=0.4,fill=gray!70] 
($(1) + (0,0.34)$) to[out=180,in=90] ($(1) + (-0.34,0)$)  to[out=270,in=180] ($(1) + (0,-0.34)$)  to ($(4) + (0,-0.34)$) to[out=0,in=270] ($(4) + (0.34,0)$) to[out=90,in=0] ($(4) + (0,0.34)$) to ($(1) + (0,0.34)$);

\filldraw[fill opacity=0,fill=white!70] 
($(6) + (-0.2,0.2)$) to[out=225,in=135] ($(6) + (-0.2,-0.2)$)  to[out=315,in=225] ($(6) + (0.2,-0.2)$)  to ($(1) + (0.2,-0.2)$) to[out=45,in=305] ($(1) + (0.2,0.2)$) to[out=135,in=45] ($(1) + (-0.2,0.2)$) to ($(6) + (-0.2,0.2)$);

\filldraw[fill opacity=0,fill=white!70] 
($(2) + (0.4,0)$) to[out=90,in=0] ($(2) + (0,0.4)$)  to[out=180,in=90] ($(2) + (-0.4,0)$)  to ($(8) + (-0.4,0)$) to[out=270,in=180] ($(8) + (0,-0.4)$) to[out=0,in=270] ($(8) + (0.4,0)$) to ($(2) + (0.4,0)$);

\filldraw[fill opacity=0,fill=white!70] 
($(3) + (0.4,0)$) to[out=90,in=0] ($(3) + (0,0.4)$)  to[out=180,in=90] ($(3) + (-0.4,0)$)  to ($(10) + (-0.4,0)$) to[out=270,in=180] ($(10) + (0,-0.4)$) to[out=0,in=270] ($(10) + (0.4,0)$) to ($(3) + (0.4,0)$);

\filldraw[fill opacity=0,fill=white!70] 
($(4) + (-0.2,-0.2)$) to[out=135,in=225] ($(4) + (-0.2,0.2)$)  to[out=45,in=135] ($(4) + (0.2,0.2)$)  to ($(12) + (0.2,0.2)$) to[out=305,in=45] ($(12) + (0.2,-0.2)$) to[out=225,in=305] ($(12) + (-0.2,-0.2)$) to ($(4) + (-0.2,-0.2)$);

\filldraw[fill opacity=0.4,fill=gray!70] 
($(20) + (0.4,0)$) to[out=90,in=0] ($(20) + (0,0.4)$)  to[out=180,in=90] ($(20) + (-0.4,0)$)  to ($(11) + (-0.4,0)$) to[out=270,in=180] ($(11) + (0,-0.4)$) to[out=0,in=270] ($(11) + (0.4,0)$) to ($(20) + (0.4,0)$);

\filldraw[fill opacity=0,fill=white!70] 
($(13) + (0,0.34)$) to[out=180,in=90] ($(13) + (-0.34,0)$)  to[out=270,in=180] ($(13) + (0,-0.34)$)  to ($(17) + (0,-0.34)$) to[out=0,in=270] ($(17) + (0.34,0)$) to[out=90,in=0] ($(17) + (0,0.34)$) to ($(13) + (0,0.34)$);

\filldraw[fill opacity=0,fill=white!70] 
($(14) + (0,0.34)$) to[out=180,in=90] ($(14) + (-0.34,0)$)  to[out=270,in=180] ($(14) + (0,-0.34)$)  to ($(19) + (0,-0.34)$) to[out=0,in=270] ($(19) + (0.34,0)$) to[out=90,in=0] ($(19) + (0,0.34)$) to ($(14) + (0,0.34)$);

\filldraw[fill opacity=0,fill=white!70] 
($(15) + (-0.2,0.2)$) to[out=225,in=135] ($(15) + (-0.2,-0.2)$)  to[out=315,in=225] ($(15) + (0.2,-0.2)$)  to ($(21) + (0.2,-0.2)$) to[out=45,in=305] ($(21) + (0.2,0.2)$) to[out=135,in=45] ($(21) + (-0.2,0.2)$) to ($(15) + (-0.2,0.2)$);

\filldraw[fill opacity=0.4,fill=gray!70] 
($(24) + (0,0.34)$) to[out=180,in=90] ($(24) + (-0.34,0)$)  to[out=270,in=180] ($(24) + (0,-0.34)$)  to ($(15) + (0,-0.34)$) to[out=0,in=270] ($(15) + (0.34,0)$) to[out=90,in=0] ($(15) + (0,0.34)$) to ($(24) + (0,0.34)$);

\filldraw[fill opacity=0,fill=white!70] 
($(26) + (0.4,0)$) to[out=90,in=0] ($(26) + (0,0.4)$)  to[out=180,in=90] ($(26) + (-0.4,0)$)  to ($(22) + (-0.4,0)$) to[out=270,in=180] ($(22) + (0,-0.4)$) to[out=0,in=270] ($(22) + (0.4,0)$) to ($(26) + (0.4,0)$);

\filldraw[fill opacity=0,fill=white!70] 
($(28) + (0.4,0)$) to[out=90,in=0] ($(28) + (0,0.4)$)  to[out=180,in=90] ($(28) + (-0.4,0)$)  to ($(23) + (-0.4,0)$) to[out=270,in=180] ($(23) + (0,-0.4)$) to[out=0,in=270] ($(23) + (0.4,0)$) to ($(28) + (0.4,0)$);

\filldraw[fill opacity=0,fill=white!70] 
($(30) + (-0.2,-0.2)$) to[out=135,in=225] ($(30) + (-0.2,0.2)$)  to[out=45,in=135] ($(30) + (0.2,0.2)$)  to ($(24) + (0.2,0.2)$) to[out=305,in=45] ($(24) + (0.2,-0.2)$) to[out=225,in=305] ($(24) + (-0.2,-0.2)$) to ($(30) + (-0.2,-0.2)$);

\filldraw[fill opacity=0.4,fill=gray!70] 
($(29) + (0.4,0)$) to[out=90,in=0] ($(29) + (0,0.4)$)  to[out=180,in=90] ($(29) + (-0.4,0)$)  to ($(5) + (-0.4,0)$) to[out=270,in=180] ($(5) + (0,-0.4)$) to[out=0,in=270] ($(5) + (0.4,0)$) to ($(29) + (0.4,0)$);

\filldraw[fill opacity=0,fill=white!70] 
($(34) + (0,0.34)$) to[out=180,in=90] ($(34) + (-0.34,0)$)  to[out=270,in=180] ($(34) + (0,-0.34)$)  to ($(31) + (0,-0.34)$) to[out=0,in=270] ($(31) + (0.34,0)$) to[out=90,in=0] ($(31) + (0,0.34)$) to ($(34) + (0,0.34)$);

\filldraw[fill opacity=0,fill=white!70] 
($(36) + (0,0.34)$) to[out=180,in=90] ($(36) + (-0.34,0)$)  to[out=270,in=180] ($(36) + (0,-0.34)$)  to ($(32) + (0,-0.34)$) to[out=0,in=270] ($(32) + (0.34,0)$) to[out=90,in=0] ($(32) + (0,0.34)$) to ($(36) + (0,0.34)$);

\node at (3.5,2.6) {$b_1$};
\node at (0.2,1) {$a_1$};
\node at (2.3,1) {$a_2$};
\node at (4.7,1) {$a_3$};
\node at (6.2,0) {$a_4$};
\node at (6.7,1.7) {$b_2$};
\node at (4.3,3) {$a_5$};
\node at (4.3,4) {$a_6$};
\node at (6.2,7) {$a_7$};
\node at (3.5,4.4) {$b_3$};
\node at (4.7,6) {$a_8$};
\node at (2.3,6) {$a_9$};
\node at (0.8,7) {$a_{10}$};
\node at (0.4,5.4) {$b_4$};
\node at (2.7,4) {$a_{11}$};
\node at (2.7,3) {$a_{12}$};
\end{tikzpicture}
\end{center}
\caption{An example of an $AB$-cycle of length $4$, with $k_A=3$ and $k_B=4$. 
The $B$-edges $b_1,b_2,b_3,b_4$ are shaded in grey.
}
\label{figAB2}
%
%
%
\begin{center}
\begin{tikzpicture}[scale=0.9]
\node[vertex] at (6,2) (1) {};
\node[vertex] at (7,2) (2) {};
\node[vertex] at (8,2) (3) {};
\node[vertex] at (9,2) (4) {};
\node[vertex] at (5,1) (5) {};
\node[vertex] at (4,0) (6) {};
\node[vertex] at (7,1) (7) {};
\node[vertex] at (7,0) (8) {};
\node[vertex] at (8,1) (9) {};
\node[vertex] at (8,0) (10) {};
\node[vertex] at (10,1) (11) {};
\node[vertex] at (11,0) (12) {};
\node[vertex] at (9,3) (13) {};
\node[vertex] at (9,4) (14) {};
\node[vertex] at (9,5) (15) {};
\node[vertex] at (10,3) (16) {};
\node[vertex] at (11,3) (17) {};
\node[vertex] at (10,4) (18) {};
\node[vertex] at (11,4) (19) {};
\node[vertex] at (10,6) (20) {};
\node[vertex] at (11,7) (21) {};
\node[vertex] at (8,5) (22) {};
\node[vertex] at (7,5) (23) {};
\node[vertex] at (6,5) (24) {};
\node[vertex] at (8,6) (25) {};
\node[vertex] at (8,7) (26) {};
\node[vertex] at (7,6) (27) {};
\node[vertex] at (7,7) (28) {};
\node[vertex] at (5,6) (29) {};
\node[vertex] at (4,7) (30) {};
\node[vertex] at (6,4) (31) {};
\node[vertex] at (6,3) (32) {};
\node[vertex] at (5,4) (33) {};
\node[vertex] at (4,4) (34) {};
\node[vertex] at (5,3) (35) {};
\node[vertex] at (4,3) (36) {};

\node[vertex] at (0,7) (37) {};
\node[vertex] at (1,7) (38) {};
\node[vertex] at (2,7) (39) {};
\node[vertex] at (0,6) (40) {};
\node[vertex] at (1,6) (41) {};
\node[vertex] at (2,6) (42) {};
\node[vertex] at (0,5) (43) {};
\node[vertex] at (1,5) (44) {};
\node[vertex] at (2,5) (45) {};
\node[vertex] at (0,4) (46) {};
\node[vertex] at (1,4) (47) {};
\node[vertex] at (2,4) (48) {};
\node[vertex] at (0,3) (49) {};
\node[vertex] at (1,3) (50) {};
\node[vertex] at (2,3) (51) {};
\node[vertex] at (1,2) (52) {};
\node[vertex] at (2,2) (53) {};
\node[vertex] at (3,2) (54) {};

\filldraw[fill opacity=0.4,fill=gray!70] 
($(1) + (0,0.34)$) to[out=180,in=90] ($(1) + (-0.34,0)$)  to[out=270,in=180] ($(1) + (0,-0.34)$)  to ($(4) + (0,-0.34)$) to[out=0,in=270] ($(4) + (0.34,0)$) to[out=90,in=0] ($(4) + (0,0.34)$) to ($(1) + (0,0.34)$);

\filldraw[fill opacity=0,fill=white!70] 
($(6) + (-0.2,0.2)$) to[out=225,in=135] ($(6) + (-0.2,-0.2)$)  to[out=315,in=225] ($(6) + (0.2,-0.2)$)  to ($(1) + (0.2,-0.2)$) to[out=45,in=305] ($(1) + (0.2,0.2)$) to[out=135,in=45] ($(1) + (-0.2,0.2)$) to ($(6) + (-0.2,0.2)$);

\filldraw[fill opacity=0,fill=white!70] 
($(2) + (0.4,0)$) to[out=90,in=0] ($(2) + (0,0.4)$)  to[out=180,in=90] ($(2) + (-0.4,0)$)  to ($(8) + (-0.4,0)$) to[out=270,in=180] ($(8) + (0,-0.4)$) to[out=0,in=270] ($(8) + (0.4,0)$) to ($(2) + (0.4,0)$);

\filldraw[fill opacity=0,fill=white!70] 
($(3) + (0.4,0)$) to[out=90,in=0] ($(3) + (0,0.4)$)  to[out=180,in=90] ($(3) + (-0.4,0)$)  to ($(10) + (-0.4,0)$) to[out=270,in=180] ($(10) + (0,-0.4)$) to[out=0,in=270] ($(10) + (0.4,0)$) to ($(3) + (0.4,0)$);

\filldraw[fill opacity=0,fill=white!70] 
($(4) + (-0.2,-0.2)$) to[out=135,in=225] ($(4) + (-0.2,0.2)$)  to[out=45,in=135] ($(4) + (0.2,0.2)$)  to ($(12) + (0.2,0.2)$) to[out=305,in=45] ($(12) + (0.2,-0.2)$) to[out=225,in=305] ($(12) + (-0.2,-0.2)$) to ($(4) + (-0.2,-0.2)$);

\filldraw[fill opacity=0.4,fill=gray!70] 
($(20) + (0.4,0)$) to[out=90,in=0] ($(20) + (0,0.4)$)  to[out=180,in=90] ($(20) + (-0.4,0)$)  to ($(11) + (-0.4,0)$) to[out=270,in=180] ($(11) + (0,-0.4)$) to[out=0,in=270] ($(11) + (0.4,0)$) to ($(20) + (0.4,0)$);

\filldraw[fill opacity=0,fill=white!70] 
($(13) + (0,0.34)$) to[out=180,in=90] ($(13) + (-0.34,0)$)  to[out=270,in=180] ($(13) + (0,-0.34)$)  to ($(17) + (0,-0.34)$) to[out=0,in=270] ($(17) + (0.34,0)$) to[out=90,in=0] ($(17) + (0,0.34)$) to ($(13) + (0,0.34)$);

\filldraw[fill opacity=0,fill=white!70] 
($(14) + (0,0.34)$) to[out=180,in=90] ($(14) + (-0.34,0)$)  to[out=270,in=180] ($(14) + (0,-0.34)$)  to ($(19) + (0,-0.34)$) to[out=0,in=270] ($(19) + (0.34,0)$) to[out=90,in=0] ($(19) + (0,0.34)$) to ($(14) + (0,0.34)$);

\filldraw[fill opacity=0,fill=white!70] 
($(15) + (-0.2,0.2)$) to[out=225,in=135] ($(15) + (-0.2,-0.2)$)  to[out=315,in=225] ($(15) + (0.2,-0.2)$)  to ($(21) + (0.2,-0.2)$) to[out=45,in=305] ($(21) + (0.2,0.2)$) to[out=135,in=45] ($(21) + (-0.2,0.2)$) to ($(15) + (-0.2,0.2)$);

\filldraw[fill opacity=0.4,fill=gray!70] 
($(24) + (0,0.34)$) to[out=180,in=90] ($(24) + (-0.34,0)$)  to[out=270,in=180] ($(24) + (0,-0.34)$)  to ($(15) + (0,-0.34)$) to[out=0,in=270] ($(15) + (0.34,0)$) to[out=90,in=0] ($(15) + (0,0.34)$) to ($(24) + (0,0.34)$);

\filldraw[fill opacity=0,fill=white!70] 
($(26) + (0.4,0)$) to[out=90,in=0] ($(26) + (0,0.4)$)  to[out=180,in=90] ($(26) + (-0.4,0)$)  to ($(22) + (-0.4,0)$) to[out=270,in=180] ($(22) + (0,-0.4)$) to[out=0,in=270] ($(22) + (0.4,0)$) to ($(26) + (0.4,0)$);

\filldraw[fill opacity=0,fill=white!70] 
($(28) + (0.4,0)$) to[out=90,in=0] ($(28) + (0,0.4)$)  to[out=180,in=90] ($(28) + (-0.4,0)$)  to ($(23) + (-0.4,0)$) to[out=270,in=180] ($(23) + (0,-0.4)$) to[out=0,in=270] ($(23) + (0.4,0)$) to ($(28) + (0.4,0)$);

\filldraw[fill opacity=0,fill=white!70] 
($(30) + (-0.2,-0.2)$) to[out=135,in=225] ($(30) + (-0.2,0.2)$)  to[out=45,in=135] ($(30) + (0.2,0.2)$)  to ($(24) + (0.2,0.2)$) to[out=305,in=45] ($(24) + (0.2,-0.2)$) to[out=225,in=305] ($(24) + (-0.2,-0.2)$) to ($(30) + (-0.2,-0.2)$);

\filldraw[fill opacity=0.4,fill=gray!70] 
($(29) + (0.4,0)$) to[out=90,in=0] ($(29) + (0,0.4)$)  to[out=180,in=90] ($(29) + (-0.4,0)$)  to ($(5) + (-0.4,0)$) to[out=270,in=180] ($(5) + (0,-0.4)$) to[out=0,in=270] ($(5) + (0.4,0)$) to ($(29) + (0.4,0)$);

\filldraw[fill opacity=0,fill=white!70] 
($(34) + (0,0.34)$) to[out=180,in=90] ($(34) + (-0.34,0)$)  to[out=270,in=180] ($(34) + (0,-0.34)$)  to ($(31) + (0,-0.34)$) to[out=0,in=270] ($(31) + (0.34,0)$) to[out=90,in=0] ($(31) + (0,0.34)$) to ($(34) + (0,0.34)$);

\filldraw[fill opacity=0,fill=white!70] 
($(36) + (0,0.34)$) to[out=180,in=90] ($(36) + (-0.34,0)$)  to[out=270,in=180] ($(36) + (0,-0.34)$)  to ($(32) + (0,-0.34)$) to[out=0,in=270] ($(32) + (0.34,0)$) to[out=90,in=0] ($(32) + (0,0.34)$) to ($(36) + (0,0.34)$);

\filldraw[fill opacity=0.4,fill=gray!70] 
($(37) + (0,0.34)$) to[out=180,in=90] ($(37) + (-0.34,0)$)  to[out=270,in=180] ($(37) + (0,-0.34)$)  to ($(30) + (0,-0.34)$) to[out=0,in=270] ($(30) + (0.34,0)$) to[out=90,in=0] ($(30) + (0,0.34)$) to ($(37) + (0,0.34)$);

\filldraw[fill opacity=0,fill=white!70] 
($(46) + (0,0.34)$) to[out=180,in=90] ($(46) + (-0.34,0)$)  to[out=270,in=180] ($(46) + (0,-0.34)$)  to ($(48) + (0,-0.34)$) to[out=0,in=270] ($(48) + (0.34,0)$) to[out=90,in=0] ($(48) + (0,0.34)$) to ($(46) + (0,0.34)$);

\filldraw[fill opacity=0,fill=white!70] 
($(49) + (0,0.34)$) to[out=180,in=90] ($(49) + (-0.34,0)$)  to[out=270,in=180] ($(49) + (0,-0.34)$)  to ($(51) + (0,-0.34)$) to[out=0,in=270] ($(51) + (0.34,0)$) to[out=90,in=0] ($(51) + (0,0.34)$) to ($(49) + (0,0.34)$);

\filldraw[fill opacity=0,fill=white!70] 
($(52) + (0,0.34)$) to[out=180,in=90] ($(52) + (-0.34,0)$)  to[out=270,in=180] ($(52) + (0,-0.34)$)  to ($(54) + (0,-0.34)$) to[out=0,in=270] ($(54) + (0.34,0)$) to[out=90,in=0] ($(54) + (0,0.34)$) to ($(52) + (0,0.34)$);

\filldraw[fill opacity=0,fill=white!70] 
($(37) + (0.4,0)$) to[out=90,in=0] ($(37) + (0,0.4)$)  to[out=180,in=90] ($(37) + (-0.4,0)$)  to ($(43) + (-0.4,0)$) to[out=270,in=180] ($(43) + (0,-0.4)$) to[out=0,in=270] ($(43) + (0.4,0)$) to ($(37) + (0.4,0)$);

\filldraw[fill opacity=0,fill=white!70] 
($(38) + (0.4,0)$) to[out=90,in=0] ($(38) + (0,0.4)$)  to[out=180,in=90] ($(38) + (-0.4,0)$)  to ($(44) + (-0.4,0)$) to[out=270,in=180] ($(44) + (0,-0.4)$) to[out=0,in=270] ($(44) + (0.4,0)$) to ($(38) + (0.4,0)$);

\filldraw[fill opacity=0,fill=white!70] 
($(39) + (0.4,0)$) to[out=90,in=0] ($(39) + (0,0.4)$)  to[out=180,in=90] ($(39) + (-0.4,0)$)  to ($(45) + (-0.4,0)$) to[out=270,in=180] ($(45) + (0,-0.4)$) to[out=0,in=270] ($(45) + (0.4,0)$) to ($(39) + (0.4,0)$);

\filldraw[fill opacity=0.4,fill=gray!70] 
($(43) + (-0.2,-0.2)$) to[out=135,in=225] ($(43) + (-0.2,0.2)$)  to[out=45,in=135] ($(43) + (0.2,0.2)$)  to ($(54) + (0.2,0.2)$) to[out=305,in=45] ($(54) + (0.2,-0.2)$) to[out=225,in=305] ($(54) + (-0.2,-0.2)$) to ($(43) + (-0.2,-0.2)$);

\node at (7.5,2.6) {$b_1$};
\node at (4.2,1) {$a_1$};
\node at (6.3,1) {$a_2$};
\node at (8.7,1) {$a_3$};
\node at (10.2,0) {$a_4$};
\node at (10.7,1.7) {$b_2$};
\node at (8.3,3) {$a_5$};
\node at (8.3,4) {$a_6$};
\node at (10.2,7) {$a_7$};
\node at (7.5,4.4) {$b_3$};
\node at (8.7,6) {$a_8$};
\node at (6.3,6) {$a_9$};
\node at (4.8,7) {$a_{10}$};
\node at (4.4,5.4) {$b_4$};
\node at (6.7,4) {$a_{11}$};
\node at (6.7,3) {$a_{12}$};
\node at (3,7.7) {$b_5$};
\node at (2,7.7) {$a_{13}$};
\node at (1,7.7) {$a_{14}$};
\node at (-0.7,6) {$a_{15}$};
\node at (2.85,2.85) {$b_6$};
\node at (-0.7,4) {$a_{16}$};
\node at (-0.7,3) {$a_{17}$};
\node at (0.3,2) {$a_{18}$};
\node at (4,7.5) {$v$};
\end{tikzpicture}
\end{center}
\caption{An example of an $AB$-cycle-path with parameters $s=4$, $t=2$, with $k_A=3$ and $k_B=4$. 
The $B$-edges $b_1,\dots,b_6$ are shaded in grey.
Observe that $S:=\{a_1,\dots,a_{12},b_1,\dots,b_4\}$ forms an $AB$-cycle of length $s=4$
and $T:=\{a_{10},a_{13},\dots,a_{18},b_5,b_6\}$ forms an $AB$-path of length $t=2$,
where $V(S) \cap V(T)=a_{10}$, $E(S) \cap E(T)=\{a_{10}\}$ and 
the vertex $v:=a_{10} \cap b_5$ does not lie in any of the $B$-edges (i.e. $b_1,b_2,b_3, b_4$) of $S$.
}
\label{figAB3}
\end{figure}

\subsection{The deterministic and probabilistic lemmas}
The following two rather technical lemmas immediately  combine to ensure 
w.h.p. $H$ is empty.
\begin{lemma}[Deterministic lemma]\label{detclaim}
If $H$ is non-empty then it contains at least one of the following structures:
\begin{itemize}
    \item[(i)] An $AB$-path of length at least $\log n$.
    \item[(ii)] An $A$-path of length at least $ \log n$.
    \item[(iii)] Two $A$-edges that intersect in at least $2$ vertices.
    \item[(iv)] An $A$-cycle of length at most $1+\log n$.
    \item[(v)] 
    A $B$-edge $b$ with vertex set $v_1,\dots, v_{k_B}$, together with $A$-edges $a_1,\dots,a_{k_B}$ where $a_i \cap b =\{v_i\}$ (for 
    all $i\in[k_B]$) so that 
    \begin{itemize}
        \item there exist $a_i$ and $a_j$ that intersect and
        \item in the  edge order  $a_1,\dots,a_{k_B}$, 
 for  each $i\geq 2$, $a_i$ has at most one old vertex.
    \end{itemize}
    \item[(vi)] An $AB$-set $S$ together with an $A$-edge $e$ that intersects the $B$-edge $b$ of $S$ in at least $2$ vertices, but $e$ intersects each $A$-edge in $S$ in at most one vertex.
    \item[(vii)] An $AB$-set $S$ (consisting of a $B$-edge $b$, and $A$-edges
    $a_1,\dots, a_{k_B}$) and a collection of $A$-edges $e_1,\dots, e_s$ where
    \begin{itemize}
        \item $1\leq s \leq \log n$;
        \item $a_1,e_1,\dots, e_s,a_2$ forms an $A$-path that only
        intersects $b$ in the vertices $v_1:=a_1 \cap b$ and $v_2:=a_2 \cap b$;
        \item for each $i\geq 3$, $a_i$ intersects the $A$-path 
        $e_1,\dots, e_s$ in at most one vertex.
    \end{itemize}
    \item[(viii)] An $AB$-cycle-path $P$ with parameters $s,t \leq \log n$ together with 
    an $A$-edge $a$ such that $2 \leq |a \cap V(P)| \leq k_A-1$.
    \COMMENT{AT replaced $P$ with $V(P)$ here... and later on in paper too}
    \item[(ix)] An $AB$-cycle-path $P$ with parameters $s,t \leq \log n$
    together with an additional $B$-edge $b$ and additional $A$-edges $a_1,\dots,a_{q}$ (for some $0\leq q \leq k_B-1$) such that
    \begin{itemize}
        \item there exist $A$-edges $a_{q+1},\dots,a_{k_B}$ from $P$ so that $a_1,\dots,a_{k_B}$ together with $b$ form an $AB$-set;
        \item each $A$-edge $a_i$, $i \in [q]$, intersects $P$ in at most one vertex; 
        \item we have that the vertex $v:=b \cap a_{k_B}$ lies in no $B$-edge of $P$. 
        \COMMENT{RH: The intuition here will be that we are trying to extend the $AB$-cycle-path using the vertex
        $v=b \cap a_{k_B}$ but fail; either we intersect too many previous old vertices (condition on $q$ below) or there is a new $A$-edge
        which intersects an old $A$-edge (condition of $a_i$ intersecting $P$ below).
        I've not specified that $a_{k_B}$ should be at the end of the path... but this is okay, since it just means we have ruled more stuff!}
    \end{itemize}
    Further, at least one of the following holds:
    \begin{itemize}
        \item $s \geq 2$ and $q \leq k_B-2$;
        \item $s=0$ and $q \leq k_B-3$;
        \item there exists $i \in [q]$ such that $a_i$ intersects $P$.
        \end{itemize}
    \end{itemize}
\end{lemma}

\begin{lemma}[Probabilistic lemma]\label{proclaim}
W.h.p. $G$ (and therefore $H$) does not contain any of the structures described by (i)--(ix) in 
Lemma~\ref{detclaim}.
\end{lemma}

In the next subsection we prove the deterministic lemma, followed by a proof of the 
probabilistic lemma, thereby completing the proof of Theorem~\ref{mainthm} in the case when $k_B>k_A$.
We note that the condition $k_B>k_A$ is needed within the proof of the probabilistic lemma, see calculations~(\ref{eq:prob1})--(\ref{eq:prob2}).

\subsection{Proof of Lemma~\ref{detclaim}}
Suppose for a contradiction that $H$ is non-empty but does not contain any of the
structures defined in (i)--(ix).
As there are no structures as in (ii), (iii) and (iv), this immediately implies the following.
\begin{claim}\label{cl:Atrees1}
The $A$-edges of $H$ form vertex-disjoint $A$-trees.
\end{claim}

Next we prove the following claim.
\begin{claim}\label{cl:Btrees1}
For a $B$-edge $b$ of $H$, the vertices of $b$ are each in different $A$-trees.
\end{claim}
\proof
Note by Claim~\ref{cl:determ} every vertex in $b$ lies in its own $A$-edge; 
label the vertices of $b$ by $v_1,\dots,v_{k_B}$ and their respective $A$-edges $a_1,\dots,a_{k_B}$. 
Assume for a contradiction that Claim~\ref{cl:Btrees1} does not hold for $b$.
This implies that 
there is an $A$-tree in $H$ which contains at least two of the $A$-edges
$a_1,\dots,a_{k_B}$. 
We now split into three cases:
 (a) there exists $a_i$ and $a_j$ that intersect; (b)
 there is an $A$-edge $e_a$ that intersects $b$ in $s \geq 2$ vertices, but  the edges $a_1,\dots,a_{k_B}$ are pairwise disjoint; (c)
  all $A$-edges in $H$ intersect $b$  in at most one vertex and the edges $a_1,\dots,a_{k_B}$ are pairwise disjoint.
  
 We will show that in each case we get a contradiction (i.e., we obtain one of the structures defined in (i)--(ix)). 
 First suppose (a) holds.
 By Claim~\ref{cl:Atrees1} and by definition of an $A$-tree, there exists an edge order 
(w.l.o.g. we may assume this order is $a_1,\dots,a_{k_B}$) 
of the $A$-edges so that for  each $i\geq 2$, $a_i$ has at most one old vertex.
Then $b,a_1,\dots,a_{k_B}$ together form a structure as in (v), a contradiction.

Next suppose that (b) holds.
 In this case $b$ and $a_1,\dots,a_{k_B}$ together form an $AB$-set $S$.
 Further, by Claim~\ref{cl:Atrees1}, $|e_a \cap a_i| \leq 1$ for each $i \in [k_B]$.
 So $S$ together with $e_a$ forms a structure as in (vi), a contradiction.
 
Finally suppose that (c) holds. Again in this case $b$ and $a_1,\dots,a_{k_B}$ together form an $AB$-set $S$.
Since the trees of at least two of the $A$-edges $a_1,\dots, a_{k_B}$ intersect, we may assume  with loss of generality that $a_1$ and $a_2$ lie in the same $A$-tree. 
Then consider the $A$-path $a_1, e_1,\dots,e_s, a_2 $ between $a_1$ and $a_2$ 
on this $A$-tree (where $s \geq 1$). Note that we may assume that
this $A$-path does not contain any vertices 
from $b$ (except $v_1 \in a_1$ and $v_2 \in a_2$). As (ii) does not hold
we have $s \leq \log n$.
For each  $i \geq 3$, $a_i$ intersects
the 
path $e_1,\dots,e_s$ in at most one vertex (else we would have a contradiction to
Claim~\ref{cl:Atrees1}). 
The structure described is precisely as in (vii), a contradiction.
\endproof

We now split into two cases.
In both cases we will do an edge-revealing process for $H$, starting with a particular subgraph of $H$.

{\it Case 1: $H$ contains an $AB$-cycle.}
We will construct a subgraph $J$ of $H$ using the following algorithm: 
initially $J$ is an $AB$-cycle $C$ in $H$ 
(note that $C$ has length at most $\log n$, otherwise it would contain an $AB$-path of length at least $\log n$, contradicting (i)).
Pick an arbitrary $A$-edge $a_{0,k_B-1}$ from $C$, 
and pick from it a vertex $v_{0,k_B-1}$ which is not yet covered by a $B$-edge.
We now repeat the following step (the whole of the next paragraph) for $i=1,2,\dots$.

\smallskip
{\bf Iterative step:}
By Claim~\ref{cl:determ} there must be a $B$-edge, $b_i$ in $H$ which covers $v_{i-1,k_B-1}$. 
Further, for each of the $q$ new vertices $v_{i,1},\dots,v_{i,q}$ of $b_i$
(i.e. those vertices in $b_i$ not currently in $J$),
by Claims~\ref{cl:determ} and~\ref{cl:Btrees1} there are disjoint $A$-edges $a_{i,1},\dots,a_{i,q}$ in $H$
so that $b_i \cap a_{i,j} = v_{i,j}$ for all $j \in [q]$. 
Add $b_i$ and $a_{i,1},\dots,a_{i,q}$ to $J$.
We terminate the algorithm if one (or both) of the following holds:
\begin{itemize}
\item We have $q \leq k_B-2$.
\item There exists some $a_{i,j}$ which intersects a previous $A$-edge of $J$.
\end{itemize}
If neither of the above holds, 
we set $v_{i,k_B-1}$ to be a vertex from $a_{i,k_B-1}$ which is not yet covered by a $B$-edge,
in preparation for the next step $i+1$.

\smallskip

\COMMENT{RH: Note that if the $B$-edge satisfies $q \leq k_B-2$ we can't terminate the algorithm without the $A$-edges, 
as the calculation does not yield $o(1)$. It does once you include the $q$ new $A$-edges too.}

Note that  the process terminates after at most $ \log n$ steps since otherwise $J$ (and so $H$) contains an $AB$-path of length $\log n$ contradicting (i).
Suppose the process terminated at step $t\leq \log n$.\COMMENT{AT: rewrote} So
the edges $b_i,a_{i,j}$ for each $i \in [t-1]$, $j \in [k_B-1]$ together with $a_{0,k_B-1}$ form an $AB$-path $Q$ of length $t-1$.
The $AB$-cycle $C$ and the $AB$-path $Q$ together form an $AB$-cycle-path $P$.
If there exists $i \in [q]$ such that $2 \leq |a_{t,i} \cap V(P)| \leq k_A-1$, then $P \cup a_{t,i}$ forms a structure exactly as in (viii),
a contradiction; so since each $a_{t,i}$ has at least one vertex ($v_{t,i}$) not in $P$, 
we get that for each $i \in [q]$, $a_{t,i}$ intersects $P$ in at most one place. 

We will now show that $J$ is precisely as described in (ix) with $s \geq 2$, a contradiction:
We have that $C \cup Q$ plays the role of $P$.
Also $b_t$, $a_{t,1}, \dots, a_{t,q}$ play the roles of $b$, $a_1, \dots, a_q$ respectively,
and $a_{t-1,k_B-1}$ plays the role of $a_{k_B}$.
Observe that if $q \leq k_B-2$, then $B$ intersects $k_B-1-q$ more vertices from $P$ as well as a vertex from $a_{k_B}$, 
and by Claim~\ref{cl:Btrees1} these vertices lie in disjoint $A$-edges within $P$; 
these $A$-edges play the roles of $a_{q+1},\dots,a_{k_B-1}$.
By Claim~\ref{cl:Btrees1} the edges playing the roles of $b, a_1,\dots,a_{k_B}$ form an $AB$-set. 
The edges playing the roles of $a_i$, $i \in [q]$, each intersect $P$ in at most one place.
We have $b_t \cap a_{t-1,k-1}$ does not lie in a $B$-edge of $P$. 
(In particular, this is the vertex $v_{t-1,k_B-1}$ which we chose at the end of step $t-1$ which was not yet covered by a $B$-edge.)
We have that $P$ is an $AB$-cycle-path with parameters $s,t-1$ where $s \geq 2$.
Finally, the conditions under which we terminated the algorithm ensures that either $q \leq k_B-2$ or there exists $i \in [q]$ such that $a_{t,i}$ intersects $P$.

\smallskip

{\it Case 2: $H$ does not contain an $AB$-cycle.}
We will construct a subgraph $J$ of $H$ using the following algorithm: 
initially $J$ is a single $A$-edge $a_{0,k_B-1}$. 
Pick from it any vertex $v_{0,k_B-1}$. (Note it is not yet covered by a $B$-edge.)
We now repeat precisely the same {\bf iterative step} as in Case 1 for $i=1,2,\dots$.

As before the process terminates at some value $t \leq \log n$. 
Again the edges $b_i,a_{i,j}$ for each $i \in [t-1]$, $j \in [k_B-1]$ together with $a_{0,k_B-1}$ form an $AB$-path $Q$ of length $t-1$.
Note that it cannot be the case that we terminated the algorithm with $b_t$ having $q=k_B-2$ new vertices
and also no $a_{t,j}$ intersecting a previous $A$-edge of $J$, since then $J$ would contain an $AB$-cycle, which contradicts the assumption of the case.  
If there exists $i \in [q]$ such that $2 \leq |a_{t,i} \cap Q| \leq k_A-1$, then $Q \cup a_{t,i}$ forms a structure exactly as in (viii) with $s=0$,
a contradiction; so since each $a_{t,i}$ has at least one vertex ($v_{t,i}$) not in $Q$, 
we get that each $a_{t,i}$, $i \in [q]$, intersects $Q$ in at most one place. 

One can now show that $J$ is precisely as described in (ix) with $s =0$, 
a contradiction:
We have that the $AB$-path $Q$ plays the role of the $AB$-cycle-path $P$;
$b_t$, $a_{t,1}, \dots, a_{t,q}$ play the roles of $b$, $a_1, \dots, a_q$ respectively;
$a_{t-1,k_B-1}$ plays the role of $a_{k_B}$;
if $q \leq k_B-2$, then the other $A$-edges which $b$ intersects from somewhere within $Q$ play the roles of $a_{q+1},\dots,a_{k_B-1}$.
The conditions of (ix) can now be checked and shown to follow almost identically to the previous case.

\smallskip

Since both cases yielded a contradiction, this completes the proof. 
\qed

\subsection{Proof of Lemma~\ref{proclaim}}

Let $K$ be the hypergraph with vertex set $[n]$,
whose edge set consists of those $k_A$-sets that
correspond to a $k_A$-distinct solution to $A$ in $[n]$ and those $k_B$-sets that correspond to a
$k_B$-distinct solution to $B$ in $[n]$. Note that both $H$ and $G$ are subhypergraphs of $K$.

\begin{claim}\label{cl:copiesofS}
Let $S$ be a subhypergraph of $K$ with a valid edge order.
Then $K$ contains at most $(k_B!)^{|E(S)|} n^{|V(S)|-|E(S)|}$ copies of $S$.
\end{claim}
\proof
Consider any fixed set $Q$ of $q<k$ vertices in $K$ (where $k=k_A$ or $k=k_B$).
Let $Z$ denote the number of edges of size $k$ in $K$ that contain $Q$.
Such an edge represents a solution $x=(x_1,\dots,x_{k_A})$ to $A$ (or a solution $y=(y_1,\dots,y_{k_B})$ to $B$),
where $q$ of the $x_i$ (or $y_i$) have already been chosen. It is straightforward to upper bound $Z$:
There are at most $k_B!$ choices for which of the variables the elements of $Q$ play the role of.
Once the role of the vertices in $Q$ are fixed,
there are at most $n$ choices for any of the other variables in the solution to $A$ (or $B$).
Moreover, 
since $A$ and $B$ are linear equations, once we have selected 
all but one vertex of an edge, 
the element corresponding to this last vertex is immediately determined.
Thus 
\begin{align}\label{eq:edgechoose}
Z \leq k_B! \cdot n^{k-q-1}.
\end{align}

One can construct a copy of $S$ in $K$ 
by going through the edges in the order given by the valid edge order.
We note that, at any stage of the process, it is the case that for an edge of size $k$, there are  $q<k$ vertices assigned elements already, for some $q \geq 0$. Thus we may repeatedly apply the inequality~(\ref{eq:edgechoose}) to bound the number of choices for each edge.
The bound on the number of copies of $S$ immediately follows. In particular, note we apply 
(\ref{eq:edgechoose}) $|E(S)|$ times.
\endproof
Let $S$ be a hypergraph with a valid edge order.
Write $\mathbb{E}_G(S)$ for the expected number of copies of  $S$ in $G$. 
By the previous claim, and the definition of $G(\subseteq K)$, we have that
 $\mathbb{E}_G(S)\leq (k_B!)^{|E(S)|} n^{|V(S)|-|E(S)|} p^{|V(S)|}$.
Thus, by definition of $p$ we obtain
\begin{align}\label{eq:exp}
\mathbb{E}_G(S)
\leq (k_B!)^{|E(S)|} c^{|V(S)|} n^{|V(S)|-|E(S)|-\frac{|V(S)| \cdot (k_A k_B - k_A - k_B)}{k_A k_B - k_A}} 
\leq c^{|V(S)|-|E(S)|} n^{\frac{ k_B \cdot |V(S)| - (k_A k_B - k_A) \cdot |E(S)|}{k_A k_B - k_A}},
\end{align}
where the last inequality follows since $c $ is sufficiently small compared to $1/k_B$.
Note that as $S$ has a valid edge order, $|V(S)|-|E(S)| \geq 0$ and so
\begin{align}\label{eq:exp2}
\mathbb{E}_G(S) \leq n^{\frac{ k_B \cdot |V(S)| - (k_A k_B - k_A) \cdot |E(S)|}{k_A k_B - k_A}}.
\end{align}

\medskip

Our aim now is to show that the expected number of copies of each structure (i)--(ix)
in $G$ is $o(1)$. Then by repeated applications of Markov's inequality we conclude that the lemma holds.

\smallskip

{\it Case (i)}
Fix $s:=\lceil \log n \rceil$.
Let $Q_s$ denote the $AB$-path of length $s$. 
Recall that there exists a valid edge order for  $Q_s$. 
Further $|V(Q_s)|=s(k_A (k_B-1))+k_A$ and $|E(Q_s)|=s k_B +1$. We obtain via~(\ref{eq:exp}) that
\begin{align}
\mathbb{E}_G(Q_s)= c^{s(k_A k_B -k_A-k_B)+k_A-1} n^{\frac{k_A}{k_Ak_B-k_A}}  \leq c^{\log n} n^2 =o(1),
\end{align}
where the last equality follows since $c $ is sufficiently small compared to $1/k_A$ and $ 1/k_B$. 
Thus, Markov's inequality implies that w.h.p. $G$ does not contain an $AB$-path
of length at least $\log n$.\footnote{Notice if we had chosen $p=Cn^{-\frac{k_Ak_B-k_A-k_B}{k_Ak_B-k_A}}$ for  $C>1$, the argument here would not work. This is the only part of the proof that we use the full force of our bound on
$p$.}

\smallskip

{\it Case (ii)}
As before set $s:=\lceil \log n \rceil$.
Let $P_s$ denote the $A$-path of length $s$. 
There exists a valid edge order for  $P_s$. 
Further $|V(P_s)|=s(k_A-1)+1$ and $|E(P_s)|=s$. We obtain via~(\ref{eq:exp}) that
\begin{align}
\mathbb{E}_G(P_s)\leq c^{s(k_A-2)+1} n^{\frac{k_B-s(k_B-k_A)}{k_Ak_B-k_A}} \leq c^{\log n} n=o(1),
\end{align}
where the last equality follows since $c $ is sufficiently small compared to $1/k_A$ and $1/k_B$.
Thus, Markov's inequality implies that w.h.p. $G$ does not contain an $A$-path
of length at least $\log n$.

\smallskip

{\it Case (iii)}
Let $T_s$ be the hypergraph consisting of $A$-edges $a$ and $a'$ with $|a \cap a'| = s\geq 2$ and let $X$ denote the total number of 
copies of $T_s$ in $G$ with $2 \leq s \leq k_A-1$. Clearly $a,a'$ is a valid edge order for  $T_s$. 
Note $|V(T_s)|=2k_A-s$ and $|E(T_s)|=2$, and so we obtain via~(\ref{eq:exp2}) that
\begin{align}\label{eq:prob1}
\mathbb{E}(X)= \sum_{s=2}^{k_A-1} \mathbb{E}_G(T_s) \leq \sum_{s=2}^{k_A-1} n^\frac{2k_A-s k_B}{k_Ak_B-k_A} \leq k_A \cdot n^\frac{-2(k_B-k_A)}{k_A k_B-k_A} = o(1).
\end{align}
Thus, Markov's inequality implies that w.h.p. $G$ does not contain any pair of $A$-edges that intersect in at least $2$ vertices.

\smallskip

{\it Case (iv)}
Let $C_s$ denote the $A$-cycle of length $s$; 
let $Y$ denote the number of copies of $C_s$ in $G$ with $3 \leq s \leq 1+ \log n$. Since $k_A \geq 3$ and 
each edge intersects at most two other edges in at most one vertex, $C_s$ has a valid edge order. 
We note $|V(C_s)|=s(k_A-1)$ and $|E(C_s)|=s$, and so we obtain via~(\ref{eq:exp2}) that
\begin{align}
\mathbb{E}(Y)= \sum_{s=3}^{1+ \log n} \mathbb{E}(C_s) \leq \sum_{s=3}^{1+ \log n} n^{\frac{-s(k_B-k_A)}{k_Ak_B-k_A}} \leq \log n \cdot n^\frac{-(k_B-k_A)}{k_Ak_B-k_A}= o(1).
\end{align}
Thus, by Markov's inequality we conclude that w.h.p. there does not exists an
$A$-cycle in $G$ of length at most $1+\log n$.

\smallskip

{\it Case (v)}
Consider a structure $S$ as in Lemma~\ref{detclaim}(v).
In the  edge order  $a_1,\dots,a_{k_B}$, 
 for  each $i\geq 2$, $a_i$ has at most one old vertex;
 write $x_i \in \{0,1\}$ for the number of vertices that $a_i$ intersects in
 $a_1,\dots, a_{i-1}$.
Let $x:=\sum x_i$ and note $x \geq 1$ since there exists some $a_i$ and $a_j$ that intersect. 
The edge order $b,a_1,\dots,a_{k_B}$ is clearly valid; there are $k_Bk_A-x$ vertices and $k_B+1$ edges  in this structure. 
 
Running over all choices of $x_i$ and all possible places for a given $A$-edge to intersect a previous $A$-edge, (\ref{eq:exp2}) implies that
the total expected number of copies of such hypergraphs $S$ in $G$ is at most
\begin{align}
\sum_{x=1}^{k_A-1} (k_A k_B)^x n^{\frac{k_A- xk_B}{k_Ak_B-k_A}} \leq (k_A k_B)^{k_A} n^\frac{-(k_B-k_A)}{k_A k_B- k_A}=o(1).
\end{align}
Therefore, Markov's inequality implies that w.h.p. no such structure exists in $G$.

\smallskip

{\it Case (vi)} Consider a structure $T$ as in Lemma~\ref{detclaim}(vi).
So $T$ consists of an $AB$-set $S$ 
(containing a $B$-edge $b$ and 
$A$-edges $a_1,\dots, a_{k_B}$) and 
an $A$-edge $e$ that intersects $b$
in at least $2$ vertices but 
each edge $a_1,\dots, a_{k_B}$ in 
at most one vertex.
Write $v_1,\dots, v_{k_B}$ for the vertices in $b$, where $v_i=b \cap a_i$.

We may assume $e \cap b= \{v_1,\dots,v_s\}$ where $s \geq 2$.
We may further assume that there is a non-negative integer $t\leq k_A-s$ so that 
$|e \cap a_{s+i}|=1$ for $i \in [t]$, and $|e \cap a_{s+i}|=0$ for $t<i\leq k_A-s$. 
(That is, $t$ encodes the number of $A$-edges from $a_1,\dots,a_{k_B}$ that $e$ intersects outside $b$.)
Note that $T$ is uniquely defined for a fixed  $s$ and $t$; so we write it as $T_{s,t}$.
 
Let $W$ denote the number of copies of all such structures $T_{s,t}$ in $G$ with $2 \leq s \leq k_A$
and $0 \leq t \leq k_A-s$. 
We note that $e_a,b,a_1,\dots,a_{k_B}$ is a valid edge order, $|V(T_{s,t})|=k_Bk_A+(k_A-s-t)$ and 
$|E(T_{s,t})|=k_B+2$. By applying~(\ref{eq:exp2}) we obtain that
\begin{align}
\mathbb{E}(W) = \sum_{s=2}^{k_A} \sum_{t=0}^{k_A-s} \mathbb{E}_G(T_{s,t}) \leq  \sum_{s=2}^{k_A} \sum_{t=0}^{k_A-s} n^{\frac{2k_A-(s+t)k_B}{k_Ak_B-k_A}} \leq k_A^2 \cdot n^\frac{-2(k_B-k_A)}{k_Ak_B-k_A}= o(1).
\end{align}
Therefore, Markov's inequality implies that w.h.p.  no structure as in (vi) occurs in $G$.

\smallskip

{\it Case (vii)}
Consider a structure $T$ as in
Lemma~\ref{detclaim}(vii).
So in particular, $a_1,e_1,\dots,
e_s,a_2$ is an $A$-path in $T$ where
$1\leq s \leq \log n$. Further.
for each  $i \geq 3$, $a_i$ intersects
the 
path $e_1,\dots,e_s$ in $ x_i \in \{0, 1\}$ vertices. Let $x:=\sum x_i$. 

The edge order  $b,a_1,e_1,\dots,e_s,a_2,a_3,\dots,a_{k_B}$ 
is clearly valid; in this structure there are $k_Bk_A-1-x+s(k_A-1)$ vertices, and $k_B+s+1$ edges. 
Thus running over all choices of $s$ and the $x_i$, and all possible places for a given $A$-edge to intersect a previous $A$-edge, (\ref{eq:exp2}) implies that
the total expected number of copies of such structures in $G$ is at most
\begin{align}
\sum_{s=1}^{\log n} \sum_{x=0}^{k_A-2} (s k_A)^x  n^{\frac{-(s+1)(k_B-k_A)-xk_B}{k_Ak_B-k_A}} \leq (\log n \cdot k_A)^{k_A} n^\frac{-2(k_B-k_A)}{k_A k_B-k_A}=o(1).
\end{align}
Therefore, Markov's inequality implies that w.h.p. $G$ does not contain any structure as in (vii).

\smallskip

{\it Case (viii)}
Consider a structure $S$ as in Lemma~\ref{detclaim}(viii). 
Since $x=|a \cap V(P)| \leq k_A-1$ the valid edge order for $P$ followed by $a$ is a valid edge order for $S$.

If $s \geq 2$ then this structure has $(s+t)(k_B-1)k_A+k_A-x$ vertices and $(s+t)k_B+1$ edges.
If $s=0$ then this structure has $k_A+t(k_B-1)k_A+k_A-x$ vertices and $tk_B+2$ edges.
Running over all choices of $s,t,x$,
 all possible places for where $a$ could intersect $P$, 
 and all possible places in the $AB$-cycle for the $AB$-path in $P$ to start from,
~(\ref{eq:exp2}) implies that the total expected number of copies of such hypergraphs in $G$ is at most
\begin{align}
& \sum_{s=2}^{\log n} \sum_{t=0}^{\log n} 
\sum_{x=2}^{k_A-1}
2(2 k_A k_B \log n)^{x+1} 
n^{\frac{k_A-xk_B}{k_A k_B - k_A}} 
+ 
\sum_{t=1}^{\log n} 
\sum_{x=2}^{k_A-1}
(2 k_A k_B \log n)^{x+1} 
n^{\frac{2k_A-xk_B}{k_A k_B -k_A}} \nonumber \\
\leq & \polylog(n) \cdot n^{\frac{-1}{k_A k_B -k_A}}
= o(1).
\end{align}
In particular, notice we multiply by $2$ in the first summation as recall that
there are $2$ different $AB$-cycles of a fixed size.\COMMENT{AT added and also changed some number in the inequality above}
 Markov's inequality implies that w.h.p. $G$ does not contain any structure as in (viii).
\smallskip

{\it Case (ix)}
Consider a structure $S$ as in Lemma~\ref{detclaim}(ix). 
We first show that $S$ has a valid edge order.\COMMENT{AT a number of small
changes scattered around this case}
Since the vertex $v:=b \cap a_{k_B}$ lies in no $B$-edge of $P$, there is a valid edge order of $P$ where $a_{k_B}$ is last. 
Use this order, then $b$, then $a_i$, $i \in [q]$. 
Since each of these $a_i$ intersect $b \cup P$ in at most two places and $k_A \geq 3$, 
this is valid edge order, unless if $q=0$.
In this case, take the same order, except reveal $b$ immediately before $a_{k_B}$. 
Since the vertex $v=b \cap a_{k_B}$ does not lie in any other $B$-edge (or $A$-edge by definition), 
$v$ is new in $b$. Further since $a_{k_B}$ previously had $k_A-1 \geq 2$ new vertices, it still has a new vertex in this edge order,
and thus this edge order is valid.

For each $i \in [q]$, let $x_i:=|a_i \cap V(P)|$ and note $x_i \in \{0,1\}$. Let $x:= \sum x_i$. 
We may assume $x_i=1$ for each $i \leq x$ and $x_i=0$ for $i \geq x+1$. 

Consider the case where $s \geq 2$. We have at least one of $q \leq k_B-2$ or $x \geq 1$.
The number of vertices in this structure is $(s+t)(k_B-1)k_A+qk_A-x$.
The number of edges in this structure is $(s+t)k_B+1+q$.
Running over all choices of $s,t,x,q$, 
all possible places for a given $A$-edge $a_i$, $i \leq x$, to intersect a previous $A$-edge, 
all possible choices of $k_B-q$ vertices from $P$ for $b$ to intersect
and all possible places in the $AB$-cycle for the $AB$-path to start from\COMMENT{RH: think this is needed as defined/ might as well be careful! AT good point... I've tweaked (viii) to reflect this too},
~(\ref{eq:exp2}) implies that the total expected number of copies of such hypergraphs in $G$ is at most
\begin{align}
& \sum_{s=2}^{\log n} \sum_{t=0}^{\log n} \left( \sum_{q=0}^{k_B-2} \sum_{x=0}^{q} + \sum_{q=k_B-1} \sum_{x=1}^{q} \right)
2(2 k _A k_B \log n)^{x+k_B-q+1} n^{\frac{(q+1-k_B) k_A - xk_B}{k_A k_B - k_A}} \nonumber \\
\leq & \polylog(n) \cdot n^{\frac{-1}{k_A k_B -k_A}}
= o(1).
\end{align}
\COMMENT{RH: Have this sum with only one term $q=k_B-1$ in it... 
we want to sum over $0 \leq q \leq k_B-1$ and $0 \leq x \leq q$, but exclude the case where $q=k_B-1$ and $x=0$.
I thought this was the cleanest way to write it, but let me know if you have any other suggestions? AT: looks fine. }

Now consider the case where $s=0$. We have at least one of $q \leq k-3$ or $x \geq 1$.
The number of vertices in this structure is $k_A+t(k_B-1)k_A+qk_A-x$.
The number of edges in this structure is $1+tk_B+1+q$.
Again running over all choices of $t,x,q$, 
all possible places for a given $A$-edge $a_i$, $i \leq x$, to intersect a previous $A$-edge,
all possible choices of $k_B-q$ vertices from $P$ for $b$ to intersect,
~(\ref{eq:exp2}) implies that the total expected number of copies of such hypergraphs in $G$ is at most
\begin{align}\label{eq:prob2}
& \sum_{t=1}^{\log n} \left( \sum_{q=0}^{k_B-3} \sum_{x=0}^{q} + \sum_{q=k_B-2}^{k_B-1} \sum_{x=1}^{q} \right)
(2 k_A k_B \log n)^{x+k_B-q} n^{\frac{(q+2-k_B) k_A - xk_B}{k_A k_B - k_A}} \nonumber \\
\leq & \polylog(n) \cdot n^{\frac{-1}{k_A k_B -k_A}}
= o(1).
\end{align}
Therefore, by Markov's inequality implies w.h.p. no such structures $S$ (with $s \geq 2$ or $s=0$) exist in $G$.
\qed

\section{Proof of Theorem~\ref{mainthm2} and the $k_A=k_B$ case of Theorem~\ref{mainthm}}\label{sec:mainproof2}

\subsection{Overview of the argument in~\cite{random4}}\label{31}
The original proof of Theorem~\ref{radores0} considers
 an analogous hypergraph $G$ to that considered in Theorem~\ref{mainthm},
 and its minimal Rado subgraph $H$.
 That is, $G$  has vertex set $[n]_p$ and edges corresponding to $k$-distinct solutions to $Ax=0$. 
 
 If $[n]_p$ is $(A,r)$-Rado then it is shown that $H$ contains a so-called \emph{spoiled simple path} or a \emph{fairly simple cycle with a handle} (see  Section~\ref{33} for these definitions).
 It is then shown that w.h.p. $G$ (and therefore $H$) has neither of these structures.
 However, the argument given in~\cite{random4} misses a case in which  neither
 of these structures has been proven to be present. To close this gap, we show that $H$ must contain at least one of these two original structures, or one of four  other structures (which we define below). We then show that w.h.p. $G$ has none of these six structures.

\subsection{A unifying theorem}\label{32}
As mentioned in the introduction, we prove Theorem~\ref{mainthm2} for a more general class of systems of linear equations. 
Let $(*)$ be the following matrix property: 
\begin{itemize}
\item[$(*)$] Under Gaussian elimination the matrix does not have any row which consists of precisely two non-zero rational entries. 
\end{itemize}
Suppose $A$ and $B$ are irredundant matrices that satisfy $(*)$.
Then Proposition 4.3(iv)-(v) in~\cite{hst} implies that 
the definitions of $m(A)$ and  $m(A,B)$  are well-defined (i.e. have positive denominator);
Proposition~12 in~\cite{hancock} implies that $C(A)$ is also well-defined and 
 satisfies $(*)$ itself.

\begin{thm}\label{mainthm2general}
Let $k, \ell$ be positive integers such that $k \geq \ell+2$.
Then there exists a constant $c>0$ such that the following holds.
Let $A$ and $B$ be systems of linear equations for which both of their underlying matrices are irredundant and satisfy $(*)$,
and their cores $C(A)$ and $C(B)$ are both of dimension $\ell \times k$. 
If $$p \leq cn^{-1/m(A,B)}=cn^{-\frac{k-\ell-1}{k-1}}$$ then  
$
\lim_{n \to \infty} \mathbb{P}[ [n]_p \text{ is  $(A,B)$-Rado}]=0.
$
\end{thm}
\COMMENT{RH: as I mentioned in an earlier email: 
For linear equations with $k \geq 3$ we observed that the random question did make sense since $\mathbb{N}$ is $(A,B)$-Rado.
However this is not true necessarily for matrices which satisfy $(*)$. 
(In particular, this is exactly what me and Christoph Spiegel had looked at and spent a bit of time on: we found a class of $2 \times 4$ matrices which satisfy $(*)$ which we could find a 2-colouring of $\mathbb{N}$ for which there are no monochromatic solutions...)}

Note that the class of matrices which are irredundant and partition regular is a subclass of the matrices which are irredundant and satisfy $(*)$ 
(as noted in Section 4.1 of~\cite{hst}), and so Theorem~\ref{mainthm2general} is indeed a generalisation of Theorem~\ref{mainthm2}. 
Also note that the underlying matrix of a linear equation of length $k \geq 3$ satisfies $(*)$, 
so Theorem~\ref{mainthm2general} covers the case of $k_A=k_B$ of Theorem~\ref{mainthm}.

\subsection{Proof of Theorem~\ref{mainthm2general}}\label{33}
Suppose that $A$ and  $B$ are as in the statement of the theorem. 
Let $c>0$ be a constant sufficiently small compared to $1/k$.
(So the choice of $c$ depends on $k$ only, and not on $A$ and $B$.)
It suffices to prove the theorem in the case when 
$p=cn^{-\frac{k-\ell-1}{k-1}}$.

Suppose $A$ and $B$ have dimensions $\ell_A \times k_A$ and $\ell_B \times k_B$ respectively.
By Proposition~12 in~\cite{hancock}, 
there exists vectors $a',b'$ such that 
every $k_A$-distinct solution $x=(x_1,\dots,x_{k_A})$ to $A$ contains as an ordered subvector $x'=(x_{i_1},\dots, x_{i_k})$ (where $i_1<...<i_k$), a $k$-distinct solution to $C(A)x=a'$ and also
every $k_B$-distinct solution $y=(x_1,\dots,x_{k_B})$ to $B$ contains as an ordered subvector $y'=(y_{j_1},\dots, y_{j_k})$ (where $j_1<...<j_k$), a $k$-distinct solution to $C(B)x=b'$. 
Write $A'$ for $C(A)x=a'$ and $B'$ for $C(B)x=b'$.
We consider the associated hypergraph $G=G(n,p,A',B')$ which is defined as in the proof of Theorem~\ref{mainthm}.
Note that if $[n]_p$ does not contain any red $k$-distinct solutions to $A'$ 
then it does not contain any red $k_A$-distinct solutions to $A$ by definition.
Similarly $[n]_p$ not containing any blue $k$-distinct solutions to $B'$ 
in turn implies it does not contain any blue $k_B$-distinct solutions to $B$.
Thus it suffices to show that w.h.p. $G$ is not Rado.

If $G$ is Rado, fix a Rado minimal subgraph $H$ of $G$. Otherwise set $H:=\emptyset$. 
So it suffices to prove that w.h.p. $H=\emptyset$. 

First note that Claim~\ref{cl:determ} holds as before (with $A'$ and $B'$ playing the
roles of $A$ and $B$ respectively).\COMMENT{AT added}
As in the proof of Theorem~\ref{mainthm}, we define some hypergraph notation, 
then prove the result by combining  deterministic and probabilistic lemmas. 
 
Note that in the definitions that follow, we do not care if the edges are $A'$-edges or $B'$-edges.
\begin{itemize}
\item A \emph{simple path of length $t$} ($t \in \mathbb{N}$) consists of edges $e_1,\dots,e_t$ 
such that $|e_i \cap e_j|=1$ if $j=i+1$, and $|e_i \cap e_j|=0$ if $j>i+1$. 
\item A \emph{fairly simple cycle} consists of a simple path $e_1,\dots,e_t$, $t \geq 2$, 
and an edge $e_0$ such that $|e_0 \cap e_1|=1$;  $|e_0 \cap e_i|=0$ for $2 \leq i \leq t-1$; $|e_0 \cap e_t|=s \geq 1$. 
\item A \emph{simple cycle} is a fairly simple cycle with $s=1$.
\item A simple path $P$ in $H$ is called \emph{spoiled} if it is not an induced subhypergraph of $H$, 
i.e. there is an edge $e \in E(H)$ such that $e \not \in E(P)$ and $e \subseteq V(P)$. 
\item A subhypergraph $H_0$ of $H$ is said to have a \emph{handle} if there is an edge $e$ in $H$ such that $|e| > |e \cap V(H_0)| \geq 2$.
\item  A \emph{bad triple} is set of three edges $e_1, e_x, e_y$, 
where $e_1 \cap e_x=\{x\}$, $e_1 \cap e_y=\{y\}$, $x \not=y$, and $|e_x \cap e_y| \geq 2$.
\item A \emph{Pasch configuration} is a set of four edges $e_1, e_2, e_3, e_4$ of size $3$ such that 
$v_{ij}=e_i \cap e_j$ is a distinct vertex for each pair $i<j$.
\item A \emph{faulty simple path of length $t$} ($t \geq 3$) is a simple path $e_1,\dots,e_t$ together with two edges $e_x$ and $e_z$ such that 
$e_1,e_2,e_x$ form a simple cycle with $|e_x \cap e_i|=0$ for $i \geq 3$;
$e_{t-1},e_t,e_z$ form a simple cycle with $|e_z \cap e_i|=0$ for $i \leq t-2$;
 each edge has size $3$; the edges $e_x$ and $e_z$ may or may not be disjoint.
 \COMMENT{OR: and the vertices of each $e_x$ and $e_z$ which are outside the simple path may or may not be the same vertex.}
\item A \emph{bad tight path} is a set of three edges $e_1, e_2, e_3$ each of size $3$ such that 
$|e_1 \cap e_2|=2$, $|e_1 \cap e_3|=1$ and $|e_2 \cap e_3|=2$.
\end{itemize}

\begin{lemma}[Deterministic lemma]\label{detclaim2}
If $H$ is non-empty then it contains at least one of the following structures:
\begin{itemize}
\item[(i)]  A spoiled simple path.
\item[(ii)] A fairly simple cycle with a handle.
\item[(iii)] A bad triple.
\item[(iv)] A simple path of length at least $\log n$ with edges of size $3$.
\item[(v)] A faulty simple path of length at most $\log n$.
\item[(vi)] A bad tight path.
\end{itemize}
\end{lemma}

\begin{proof}
Suppose for a contradiction that $H$ is non-empty but does not contain any of the structures defined in (i)--(vi).
Let $P=e_1, \dots, e_t$ be the longest simple path in $H$.
By Claim~\ref{cl:determ}, $t\geq 2$.
Without loss of generality assume $e_1$ is an $A'$-edge. 
Let $x,y$ be two vertices which belong only to $e_1$ in $P$, 
and let $e_x$ and $e_y$ be the two $B'$-edges of $H$ whose existence is guaranteed by Claim~\ref{cl:determ}, 
i.e. $e_z \cap e_1 = \{z\}$ for $z=x,y$. 
By the maximality of $P$, 
we have $h_z:=|V(P) \cap e_z| \geq 2$ for $z=x,y$. 

If $h_z=k$ for some $z$, then $P$ together with $e_z$ is a spoiled simple path, a contradiction. 
Otherwise, let $i_z:=\min\{i \geq 2: e_z \cap e_i \not= \emptyset\}$ for $z=x,y$, 
and assume without loss of generality that $i_y \leq i_x$. 
As we are assuming that (ii) does not hold, $e_1, \dots, e_{i_x}, e_x$ must not form a fairly simple cycle for which $e_y$ is a handle. Thus, this implies $e_y \subseteq e_1\cup \dots \cup e_{i_x}\cup e_x$. In particular, this means $e_x$ must contain all those vertices
 in $e_y$ which do not lie on $P$. In fact, this implies $e_y\cap e_x$ consists of precisely one vertex $v_{xy}$ (and $v_{xy}$ lies outside of $P$); indeed, otherwise $e_1$, $e_x$ and $e_y$ form a bad triple, a contradiction. Now consider $e_1,\dots, e_{i_y}, e_{y}$. This is a fairly
 simple cycle that $e_x$ intersects in at least two vertices (i.e. $x$ and $v_{xy}$).
 Thus, we obtain a fairly simple cycle with a handle unless all the vertices in $e_x$ lie in $e_1,\dots, e_{i_y}, e_{y}$. In particular, $e_x \subseteq (e_1 \cup e_{i_y} \cup e_y)$ as $i_y \leq i_x$. This in turn implies $e_{i_y}=e_{i_x}$. Indeed, otherwise $e_x$ must contain one vertex from $e_1$ and $k-1\geq 2$ vertices from $e_y$, a contradiction as we already observed that $e_x$ only intersects $e_y$ in one vertex.
 
 In summary, we have that
 $i_x=i_y$ and $e_x$ and $e_y$
 intersect in a single vertex $v_{xy}$ (and $v_{xy}$ lies outside of $P$).
As mentioned in the last paragraph, we must   have $e_x \subseteq (e_1 \cup e_{i_x} \cup e_y)$. 
Similarly, 
we have that $e_1,\dots,e_{i_x},e_x$ form a fairly simple cycle for which $e_y$ is a handle (a contradiction),
unless if we  have $e_y \subseteq (e_1 \cup e_{i_x} \cup e_x)$.

As $|e_x \cap e_y|=1$,
this implies $|(e_x \cap e_{i_x}) \setminus (e_1 \cup e_y)|=k-2$ and
$|(e_y \cap e_{i_x}) \setminus (e_1 \cup e_x)|=k-2$.
Moreover,
$|e_{i_x} \setminus (e_x \cup e_y)| \geq 1$; indeed, otherwise
$e_x$, $e_y$ and $e_{i_x}$ form a spoiled simple path. Recalling that
$|e_x \cap e_y\cap V(P)|=0$, altogether this gives that
 $k=|e_{i_x}| \geq 2k-3$. Thus we must have $k=3$.

If $i_x \geq 3$ then $e_1, e_x, e_y$ form a (fairly) simple cycle for which $e_{i_x}$ is a handle, a contradiction. 
Thus we have that $i_x=2$, and so $e_1, e_x, e_y, e_{i_x}$ form a Pasch configuration.

Now repeat the maximal path process which we did for $e_1$ to find $e_x$ and $e_y$, except from the other end of the path.
That is, there must exist edges $e_z$ and $e_w$ such that $e_z \cap e_t = \{z\}$, $e_w \cap e_t = \{w\}$, where $z,w$ are vertices in
$e_t$ that are not in $e_{t-1}$. 
By repeating the previous case analysis, we arrive at the conclusion that $e_{t-1}, e_t, e_z, e_w$ must also form a Pasch configuration where $e_z \cap e_w$ is a vertex $v_{zw}$ outside of $P$.

If $t \geq 3$, then $e_1,\dots,e_t, e_x, e_z$ together form a faulty simple path (i.e. one of (iv) and (v) holds, a contradiction).
Hence we must have $t=2$.

If the union of these two Pasch configurations contains $7$ vertices (i.e. $v_{xy} \not = v_{zw}$), then $e_1,e_2,e_x$ form a (fairly) simple cycle for which $e_z$ is a handle. 
So we now suppose that the two Pasch configurations cover the same $6$ vertices.
If we do not have $\{e_x, e_y\}=\{e_z, e_w\}$ then $e_x, e_z, e_y$ form a bad tight path.
Hence we do have equality and the two Pasch configurations we found are identical.
(Note that $e_x, e_y$ are $B'$-edges, whereas $e_z,e_w$ may be $A'$-edges; that is we could have edges which are both $A'$-edges and $B'$-edges.)

Relabel the edges and vertices as in the definition of a Pasch configuration.
We observe that $H$ cannot be just these four edges, even if all four edges are both $A'$-edges and $B'$-edges:
such a hypergraph is not Rado, e.g. colour $v_{12}, v_{13}, v_{34}$ red, and the remaining vertices blue.
Also, by definition of Rado minimal, this cannot be a component of $H$. That is,
there is an edge $e_5$ in $H$, where $e_5 \not= e_i$, $i \in [4]$, 
and $e_5$ contains $s$ vertices from inside the Pasch configuration, where $s \geq 1$.
If $s=1$ then w.l.o.g. $e_5$ contains $v_{1,2}$; then $e_5$, $e_1$, $e_3$ is a simple path of length $3$, 
a contradiction to the longest path in $H$ of length $2$ found earlier.
If $s=2$ then whichever $2$ vertices of the Pasch configuration $e_5$ contains, 
taking any of the simple cycles of the Pasch configuration together with $e_5$ 
gives a (fairly) simple cycle with handle.
If $s=3$, first suppose $V(e_5)=\{v_{1,2}, v_{1,3}, v_{2,3}\}$. Then $e_1,e_5,e_2$ is a bad tight path. 
If $V(e_5)=\{v_{1,2}, v_{1,3}, v_{2,4}\}$, then again $e_1, e_5, e_2$ is a bad tight path.
For all other $3$-sets of vertices $e_5$ could contain, a symmetrical argument shows that we find a bad tight path.
Since all three values of $s$ give a contradiction, this concludes the proof.
\end{proof}
The reader might wonder why we did not add the Pasch configuration to list of configurations in the statement of Lemma~\ref{detclaim2}, and then curtail our proof at the point that we conclude $H$ contains this structure: it turns out that
(e.g. if $A'$ and $B'$ correspond to $x+y=z$),  the expected number of Pasch configurations in $G$ is bounded away from $0$.
On the other hand, we now show that w.h.p. none of the structures (i)--(vi) occur in $G$.

\begin{lemma}[Probabilistic lemma]\label{proclaim2}
W.h.p. $G$ (and therefore $H$) does not contain any of the structures described by (i)--(vi) in 
Lemma~\ref{detclaim2}.
\end{lemma}

\begin{proof}
{\it Cases (i) and (ii)}
The argument in~\cite{random4} shows that w.h.p. $G$ (and therefore $H$) does not contain a spoiled simple path or a fairly simple cycle with a handle.

\smallskip

{\it Case (iii)}
Let $K$ denote the $k$-uniform hypergraph with vertex set $[n]$ 
where edges correspond to the $k$-distinct solutions to $A'$ and
 $B'$.
As in the proof of Claim~\ref{cl:copiesofS} we wish to bound the number of copies of a particular subgraph $S$ within $K$.
Suppose, similarly to the proof of Claim~\ref{cl:copiesofS}, that we are considering $Z$, 
the number of $A'$-edges and $B'$-edges of size $k$ in $K$ that contain a fixed set $Q$ of $q<k$ vertices in $K$.
In this case, by Corollary 4.6 in~\cite{hst}, we have
\begin{align}\label{eq:q1}
Z \leq \sum_{\stackrel{W \subseteq [k]}{|W|=q}} q! \cdot n^{k-q-\rank(C(A)_{\overline{W}})} + \sum_{\stackrel{W \subseteq [k]}{|W|=q}} q! \cdot n^{k-q-\rank(C(B)_{\overline{W}})}.
\end{align}

Note that if $q=|W|=2$, then by Proposition 4.3 in~\cite{hst} we have $\rank(M_{\overline{W}})=\ell$ for $M=C(A)$ and $M=C(B)$.
Thus we have 
\begin{align}\label{eq:q2}
Z \leq 2 k! \cdot n^{k-\ell-2}.
\end{align}

Now let $R_s$ be a bad triple $e_1, e_x, e_y$ with $|e_x \cap e_y|=s$ and 
let $X$ denote the total number of copies of $R_s$ in $G$ with $2 \leq s \leq k-1$.
Consider the edge order $e_x, e_y, e_1$; $e_y$ has $s$ old vertices, and $e_1$ has $2$ old vertices, and thus
we obtain via~(\ref{eq:q1}) and~(\ref{eq:q2}) that
\begin{align}\label{reee}
\mathbb{E}(X) \leq \sum_{s=2}^{k-1} \mathbb{E}_G(R_s) 
\leq 4 k!^2 n^{2k-2\ell-2} p^{2k-2} 
\left(  \sum_{\stackrel{W \subseteq [k]}{2 \leq |W| \leq k-1}} \sum_{M \in \{C(A),C(B)\}} |W|! n^{k-|W|-\rank(M_{\overline{W}})} p^{k-|W|} \right).
\end{align}

Since $C(A)$ and $C(B)$ are strictly balanced,
we may use the inequality given by~(\ref{eq:strictlyb}).
If $|W| \geq 2$, then (by e.g. Proposition 4.3(ii) in~\cite{hst}) the denominator of the
left hand side of~(\ref{eq:strictlyb}) is positive. Therefore, 
this inequality rearranges to give
\begin{align}\label{eq:q3}
\ell (k-|W|)-(k-1) \rank(M_{\overline{W}})<0,
\end{align}
where $M=C(A)$ or $M=C(B)$.
By recalling $p=cn^{-\frac{k-\ell-1}{k-1}}$, it follows that 
\begin{align}
\mathbb{E}(X) \stackrel{(\ref{reee})}{\leq} 4 k!^2 \left( \sum_{\stackrel{W \subseteq [k]}{2 \leq |W| \leq k-1}}  \sum_{M \in \{C(A),C(B)\}}  |W|! \cdot n^{k-|W|-\rank(M_{\overline{W}})} n^{-\frac{(k-\ell-1)(k-|W|)}{k-1}} \right) \stackrel{(\ref{eq:q3})}{=}o(1),
\end{align}
so by Markov's inequality we have that w.h.p.  $G$ does not contain any bad triples. 

\smallskip

For the final three cases we have $k=3$, and so we have $\ell=1$.
Then as in the proof of Theorem~\ref{mainthm}, equations~(\ref{eq:exp}) and~(\ref{eq:exp2}) hold, 
and so we may use these with $k_A=k_B=k=3$ for the remaining cases.
Note that here we have $p=cn^{-1/2}$.

{\it Case (iv)}
Fix $s:= \lceil \log n \rceil$. Let $L_s$ denote a simple path of length $s$. Recall that there exists a valid edge order for $L_s$.
Further $|V(L_s)|=2s+1$ and $|E(L_s)|=s$. We obtain via~(\ref{eq:exp}) that
\begin{align}
\mathbb{E}_G(L_s) \leq c^{\log n} n^{\frac{1}{2}} =o(1),
\end{align}
where the last equality follows since $c$ is sufficiently small compared to $1/k$. 
Thus, Markov's inequality implies that w.h.p. $G$ does not contain a simple path of length at least $\log n$.

\smallskip

{\it Case (v)}
Let $F_s$ be a faulty simple path of length $s$ and let $X$ denote the total number of copies of $F_s$ in $G$ with $3 \leq s \leq \log n$. 
Clearly $e_1,\dots,e_{s-1}, e_x,e_z, e_s$ is a valid edge order for $F_s$.
We have a choice of whether $e_x$ and $e_z$ intersect outside of the simple path or not.
If they do we obtain $|V(F_s)|=2s+2$ and if not we have $|V(F_s)|=2s+3$. In both cases we have $|E(F_s)|=s+2$, so we obtain via~(\ref{eq:exp2}) that
\begin{align}
\mathbb{E}(X)= \sum_{s=3}^{\log n} \mathbb{E}_G(F_s) \leq \sum_{s=3}^{\log n} (n^{-1} + n^{-1/2} ) \leq 2 \log n \cdot n^{-\frac{1}{2}} =o(1).
\end{align}
Thus, Markov's inequality implies that w.h.p. $G$ does not contain any faulty simple paths of length at most $\log n$.

\smallskip

{\it Case (vi)}
Let $T=e_1,e_2,e_3$ be a bad tight path. Clearly this is a valid edge order; there are $3$ edges and $5$ vertices,
and so we obtain via~(\ref{eq:exp2}) that 
$\mathbb{E}_G(T) \leq n^{-1/2} =o(1).$
Thus, Markov's inequality implies that w.h.p. $G$ does not contain $T$.
\end{proof}

\section{Concluding remarks}\label{conc}
 
\COMMENT{RH: Previously had written:
We also have a matching $0$-statement in the case where the matrices $A_1$ and $A_2$ both have cores of rank $\ell=1$ 
via our main theorem of this paper.
Not sure where this comment fits in now!}

It still remains to prove the $0$-statement of Conjecture~\ref{conj} in full generality.
One can extend the machinery we use to this general setting;
in particular, the deterministic lemma (Lemma~\ref{detclaim}) holds. However,
this does not fully resolve the $0$-statement of Conjecture~\ref{conj}
as we do not obtain a matching probabilistic lemma.
Indeed, the bound resulting from equation (\ref{eq:q1}) 
is not strong enough to conclude that (for $p$ close to the threshold given in Conjecture~\ref{conj}), 
in expectation $G$ has $o(1)$ copies of the subgraphs we wish to forbid.


As mentioned in the introduction,
it would  be interesting to deduce a matching $1$-statement for linear equations covered by Theorem~\ref{mainthm} but not by Conjecture~\ref{conj}. We believe
such a result should follow from the approach in~\cite{AP} provided one could deduce a supersaturation result of the following form: 
\begin{problem}
Let $A_1,\dots,A_r$ be systems of linear equations, 
with underlying matrices $A'_1,\dots,A'_r$ of full rank where $A'_i$ has dimension $\ell_i \times k_i$,
such that each of the $A'_i$ are irredundant,
and further $\mathbb{N}$ is $(A_1,\dots,A_r)$-Rado.
Does there exist constants $c,n_0$ such that for all $n>n_0$,
however one $r$-colours $[n]$ there exists an $i \in [r]$ such that
there are at least $c n^{k_i-\ell_i}$ solutions to $A_i$ in the $i$th colour?
\end{problem} 
Note that this would be a generalisation of the supersaturation result of Frankl, Graham and R\"odl~\cite{fgr} 
which deals with the case  where $A:=A_1=\dots=A_r$ and $A$ is a homogeneous  partition regular system of linear equations.

What about the case where one (or more) of the linear equations have only two variables?
For example, as seen in the introduction, if $A$ is $x=2y$ and $B$ is $x=4y$, then $[n]$ (for $n \geq 16$) is $(A,B)$-Rado, and so one can ask for the threshold for
$[n]_p$ being $(A,B)$-Rado. 

Finally, as pointed out by the referees, we could search for monochromatic solutions which rather than being $k$-distinct, are \emph{non-trivial}, as initially defined by Ruzsa for linear equations in~\cite{Ruzsa}, and extended to systems of linear equations in~\cite{RSZ}.
(For example, for Sidon sets where $x+y=z+w$, a solution with $z=w$ and $z\not=x$ is a non-trivial, non-$k$-distinct solution.)
It is not so natural to consider non-trivial solutions in the random setting for the symmetric case, 
and this is illustrated by the threshold given by Theorems~\ref{radores0} and~\ref{r3}.
Indeed, given an $\ell \times k$ matrix $A$, if $A$ is strictly balanced, then $n^{-1/m(A)} = n^{-(k-\ell-1)/(k-1)}$. Then at this threshold, in expectation  there are $O(n^{k-\ell-1} p^{k-1})=O(1)$ non-$k$-distinct solutions to $Ax=0$ in $[n]_p$. Therefore for $p$ significantly below this threshold,
w.h.p. $[n]_p$ contains no non-$k$-distinct solutions to  $Ax=0$.\COMMENT{AT: that's right, right? Just Markov}

In the asymmetric case, the same calculation does not hold: assuming $A$ and $B$ are linear equations with $k_B>k_A$ and $p=n^{-1/m(A,B)}$, we obtain that, in expectation, there are $\Theta(n^{k_B-2} p^{k_B-1})=\Theta(n^{k_B/k_A -1})$ non-$k_B$-distinct solutions to $B$ in $[n]_p$. Thus, it may be of interest to consider the non-trivial monochromatic solution problem in this
asymmetric setting.


\section*{Acknowledgements}
The authors are grateful to the Midlands Arts Centre for providing a nice working environment for undertaking this research, and
 to the two referees for their helpful and careful reviews.

\end{document}